\documentclass[a4wide,11pt]{article}
\usepackage[utf8]{inputenc}
\usepackage{a4wide}
\usepackage{amsmath,amsfonts,amssymb}
\usepackage{amsthm}
\usepackage{array,colortbl,xcolor,color}
\usepackage{hyperref}
\usepackage{orcidlink}

\newtheorem{theorem}{Theorem}[section]

\newtheorem{lemma}[theorem]{Lemma}

\newtheorem{definition}{Definition}[section]

\title{The turnpike property for high-dimensional interacting agent systems in discrete time}
\date{}

\author{
Martin Gugat\footnote{Chair in Dynamics, Control, Machine Learning  and  Numerics - Alexander von Humboldt Professorship, Friedrich-Alexander Universit\"at Erlangen-N\"urnberg (FAU), Cauerstr. 11, 91058 Erlangen, Germany; \texttt{martin.gugat@fau.de}}, 
$\;$
Michael Herty\footnote{Institut für Geometrie und Praktische Mathematik, RWTH Aachen University, Templergraben 55, D-52062 Aachen, Germany; \texttt{herty@igpm.rwth-aachen.de}, \texttt{segala@igpm.rwth-aachen.de}},
$\;$
Jiehong Liu\footnote{Department of Mathematics Guido Castelnuovo, Sapienza University of Rome, Piazzale Aldo Moro 5, 00185 Roma, Italy; \texttt{jiehong.liu@uniroma1.it}},
$\;$
Chiara Segala$^\dagger$
}

\begin{document}
	
\maketitle

\begin{abstract}
We investigate the interior turnpike phenomenon for discrete-time multi-agent optimal control problems.
While for continuous systems the turnpike property has been established, we focus here on  first-order discretizations of such systems. It is shown that the  resulting time-discrete system inherits the turnpike property with  estimates
of the same type  as in the continuous case. In particular,  we prove that the discrete time optimal control problem is strictly dissipative and the cheap control assumption holds.  
\end{abstract}



\section{Introduction}\label{sec:Introduction}
The collective behavior of possibly actively interacting agents, also called active particle systems, 
has been the subject of mathematical investigations
for example in \cite{MR3642940,  MR3969953}.
Most of 
the
models are governed by a system of ordinary differential equations, which can express different interaction rules between agents, such as attraction or repulsion, and in the case of active particles a control input.  Those models have been applied to several fields, such as biology, engineering, economy, and sociology, see e.g.\cite{application-[15]}\cite{application-[14]}\cite{application-[6]}\cite{application-[4]}\cite{application-[29]}. While contributions to the modeling and analysis have been numerous, the interest in control mechanism is more recent, see \cite{kinetic,sparse,sparse-Cucker}.
Typically,  the action of the control is used as  an exterior influence to enforce the emergence of the desired asymptotic target behavior of the agents. The control of high-dimensional, interacting  agents can be achieved in several ways, e.g. using Riccati-based approaches \cite{MR3431287,MR4469721}, moment-driven control \cite{MR4399019}, or model predictive control approaches \cite{MR3894072,MR3351435,MR4288153}. Those approaches have been followed to reduce the possible high computational costs.

Here, as in \cite{main}, we consider turnpike controls for such agent systems as an effective strategy to devise controls for high-dimensional systems.
The turnpike phenomenon describes the fact that,
in an interior subinterval  of the time interval, 
the dynamic control can be 
approximated to arbitrarily precision  
by the control of the (cheaper) static problem for a sufficiently large time horizon, and this reduces the computational effort as shown e.g. in  \cite{use-1,use-1.1,MR3973342}. The turnpike property has also been used to show that model-predictive control strategy converges, see e.g.  \cite{receding,use-2,MR4402854}.
While in \cite{main} the turnpike property of continuous agent systems has been investigated, we consider now the discrete-time dynamics. For discrete-time  control problems, the turnpike property has been established in \cite{turnpike-1,turnpike-2,turnpike-3,MR4599948,gruene2018turnpike}.
Contrary to those approaches, we exploit the particular structure of interacting agent systems to derive the turnpike property with interior decay \cite{interior-original, gugat2023optimal}, or the interval turnpike property \cite{MR4493557}.
\par 
The general continuous trajectories $\psi_{k}=\psi_{k}(t) \in \mathbb{R}^d$ of $k=1,\dots,N$ interacting agents in the time interval $[t_{0},T]$ are given by the following system 
\begin{equation}
\begin{aligned}
&\frac{d}{dt} \psi_{k}=\frac{1}{N}\left( \sum_{l=1}^{N}P(\psi_{k},\psi_{l})(\psi_{l}-\psi_{k}) \right)+u_{k},\quad  \psi_{k}(t_{0})=\psi_{k}^{0}.
\end{aligned}
\label{1}
\end{equation}
The data $\psi_{k}^{0} \in \mathbb{R}^d$ is the initial state of agent $k$ and $P(\cdot,\cdot)$ is a nonlinear interaction kernel representing attraction and/or repulsive forces, see e.g. \cite{attraction,pureattraction} for examples. The control $u_k=u_k(t)$ acts on each agent individually in order to minimize e.g. a  tracking-type cost towards a desired state $\bar{\psi}$ and  Tykhonov regularization for some $\gamma>0:$
\begin{equation}
J_{N}(u)=\int_{t_{0}}^{T}\frac{1}{N}\left( \sum_{k=1}^{N}\Vert\psi_{k}-\bar{\psi}\Vert^{2}+\gamma\Vert u_{k}\Vert^{2} \right) dt.
\label{2}
\end{equation}
Here and in the following we denote by $\Vert \cdot \Vert$ the Euclidean norm on $\mathbb{R}^d.$ The turnpike properties of the problem \eqref{1}--\eqref{2} have been studied in \cite{main} showing that they are independent of the number of agents $N.$ In this paper, we consider the corresponding time-discrete problem, see Section \ref{sec2} for a detailed description. In particular, we show that the turnpike property is independent of the discretization time-step $h$. Further,  the time discrete system allows for a turnpike with interior decay independent of the number of agents $N$. The proof of the interior turnpike property requires a dynamic programming principle, a strict dissipativity inequality, and the existence of a cheap control. Those results are established for the aforementioned system in Section \ref{sec3} and Section \ref{sec4}, respectively. Furthermore, the consistency in the limit for the vanishing discretization time-step is discussed in Section \ref{sec7}. Finally, numerical examples are presented in Section \ref{sec8}.

\section{The discrete-time optimal control problem}\label{sec2}

 Consider a discretization $t_{0}<t_{1}<...<t_{M}=T$ of the time interval $[t_{0},T]$, where $t_{i+1}-t_{i}=h$ for $i=0,..,M-1$ and for the stepsize
 we have $h=\frac{T-t_{0}}{M}.$ The system of ordinary differential equations \eqref{1} is discretized using an explicit Euler method. The control of agent $k$ is piecewise constant on $t_{i+1}-t_{i}$ and denoted by  $u^{i}_{k}=u_{k}$. Furthermore, let  
  $ \psi^{i}_{k}=\psi_{k}(t_{i})$ and let the initial data be given by $\psi_{k}(t_{0})$. The discrete system is then given by 
\begin{equation}
\begin{aligned}
&\psi_{k}^{i+1}=\psi_{k}^{i}+hf(\psi_{k}^{i},u_{k}^{i}), \qquad i=0,\dots,M-1, \; k=1,\dots,N,\\
&\psi_{k}^{0}=\psi_{k}(t_{0}). 
\end{aligned}
\label{3}
\end{equation}
for  
\begin{equation}
f(\psi_{k}^{i},u_{k}^{i})=\frac{1}{N}{\sum_{l=1}^{N}P(\psi_{k}^{i},\psi_{l}^{i})(\psi_{l}^{i}-\psi_{k}^{i})}+u_{k}^{i}.
\label{4}
\end{equation}
For the interaction kernel $P:\mathbb{R}^d \times \mathbb{R}^d \to \mathbb{R}$ we assume 
\begin{align}\label{assumption}
   P \in C^1(\mathbb{R}^{d\times d}), \; P(0,0) =0, \;     \| P \|_\infty \leq 0.
\end{align}
Furthermore, we denote by  $\psi^{i}=(\psi^{i}_{1},\psi^{i}_{2},\dots,\psi^{i}_{N})\in\mathbb{R}^{N\times d}$ 
the state of all agents at time  $t_i$ for $i=0,\dots,M-1$, by $\psi=\left(  \psi^i_k \right)_{i=0,\dots,M-1, k=1,\dots,N}$ and similarly for $u.$  We define the following norm of the state $\psi^i$  at the discrete time $t_i$ as 
\begin{align}\label{norm}
	\| \psi^i \|_N^2 := \sum\limits_{k=1}^N \| \psi_k^i \|^2.
\end{align}
The cost functional \eqref{2} is discretized using a first-order integration method leading to the time-discrete optimal control problem 
\begin{equation}
\begin{aligned}
Q(t_{0},T,h,N,\psi^{0})&: \,  \min_{u}J_{N}^{(h,t_{0},T)}(\psi,u), \\ \text{ for} \enspace J_{N}^{(h,t_{0},T)}(\psi,u) &:=\sum_{i=0}^{M-1} \, \frac{h}{N} 
\left(  \| \psi^i - \bar{\psi} \|^2_N + \gamma \| u^i \|^2_N
\right).
\end{aligned}
\label{5}
\end{equation}
The problem \eqref{5} is a parametric optimization problem with parameters $t_0,T$ and discretization size $h.$ In the following we assume that those parameters are chosen s.t. there exists $M \in \mathbb{N}$ with 
\begin{equation}\label{M}
    M = \frac{T-t_0}h.
\end{equation}
    
The optimal value of the  time-discrete optimal control problem is denoted by 
\begin{align}\label{v8}
    v_{N}^{(h,t_{0},T)}(\psi^{0}):= \min_{u}J_{N}^{h,t_{0},T}(\psi,u).
\end{align}
and the optimal state and control are denoted by $\hat{\psi}\in\mathbb{R}^{M\times N\times d}$ and $\hat{u}\in\mathbb{R}^{M\times N\times d},$ respectively. Further, we introduce the running cost 
\begin{align}\label{g}
    g(\psi^{i},u^{i}):=\frac{1}{N}  \left( \Vert \psi^i - \bar{\psi} \Vert^2_N + \gamma \| u^i \|_N^2
    \right).
\end{align}

Next, we establish the dynamic programming principle for the discrete in-time 
optimal control problem. 

\begin{lemma}
Let $(\hat{\psi}, \hat{u})$ be the optimal state-control pair for the optimal control problem $Q(t_{0},t_{L},h,N,\psi^{0})$ for some $t_0<t_L < T.$ 

Then,  $\hat{\psi}\mid _{(a,L)}=(\hat{\psi}^{a}, \hat{\psi}^{a+1},\dots, \hat{\psi}^{L})$ and $\hat{u}\mid _{(a,L)}=(\hat{u}^{a}, \hat{u}^{a+1},\dots, \hat{u}^{L})$ is the optimal solution to the problem $Q(t_{a},t_{L},h,N,\hat{\psi}^{a})$ for any $t_0<t_a<t_L$.
\end{lemma}

\begin{proof}
We proceed by contradiction. Note that $\psi^{i+1}_k$ is uniquely determined for given state $\psi^i_k$ and control $u^i_k$ by equation \eqref{3}.   Assume that  $\hat{\psi}\mid _{(a,L)}$ and $\hat{u}\mid _{(a,L)}$ is not the optimal solution to the problem $Q(t_{a},t_{L},h,N,\hat{\psi}^{a})$ and denote the optimal state and control by
$(\overline{\psi},\overline{u})$ where $\overline{\psi}=(\overline{\psi}^{0},\overline{\psi}^{1},\dots,\overline{\psi}^{L-a})$ and $\overline{u}=(\overline{u}^{0},\overline{u}^{1},\dots,\overline{u}^{L-a})$ for $\overline{\psi}^{0}=\hat{\psi}^{a}$. By assumption,   $(\overline{\psi},\overline{u})\not =  (\hat{\psi}\mid _{(a,L)},\hat{u}\mid _{(a,L)})$. Consider  the control $(\hat{u}\mid _{(0,a-1)}, \overline{u})$ and the corresponding state $(\hat{\psi}\mid _{(0,a-1)}, \overline{\psi})$ that has  a smaller objective value for the problem $Q(t_{0},t_{L},h,N,\psi^{0})$. This contradicts the optimality of $(\hat{\psi},\hat{u})$. 
\end{proof}   

We introduce now the
\emph{static control problem}. The static state of agent $k=1,\dots,N$ is denoted by $\psi^{(\sigma)}_k$ and the control by $u_k^{(\sigma)}$. The problem reads 
 \begin{equation}\label{10}
 \min\limits_{ \left( u_1^{(\sigma)}, \dots, u_N^{(\sigma)}  \right) } \frac{1}{N}  \left({\Vert \psi^{(\sigma)}-\bar{\psi}\Vert^{2}_N +\gamma\Vert u^{(\sigma)}\Vert_N^{2}}\right),
 \end{equation}
 subject to the stationary dynamics
 \begin{equation}\label{11}
 \frac{1}{N}
 \,\sum_{l=1}^{N}
 \left(
  P(\psi_{k}^{(\sigma)},\psi_{l}^{(\sigma)})(\psi_{l}^{(\sigma)}-\psi_{k}^{(\sigma)})+u_{k}^{(\sigma)}
 \right)
 =0.
 \end{equation}
Due to the particular structure of the agent-based dynamics, the optimal 
solution $(\psi_k^{*,(\sigma)},u^{*,(\sigma)}_k )_{k=1}^N$ to \eqref{10}--\eqref{11} is given by 
\begin{align}\label{stationary solution}
    \psi_k^{*,(\sigma)} = \bar\psi \ \mbox{ and } \ u^{*,(\sigma)}_k = 0.
 \end{align}
 Furthermore, if  we consider the dynamics \eqref{3}, the initial condition $\psi_{k}^{0}=\psi_{k}^{*,(\sigma)}$ and the control $u_{k}^{i}= u^{*,(\sigma)}_k$ for $k=1,\dots,N$, then we obtain  $\psi_{k}^{i}=\psi_{k}^{*,(\sigma)}$ for $k=1,\dots,N$ and $i \geq 1$.

In the following, we establish the turnpike property with interior decay in the sense of the  Definition \ref{def:turnpike}. We refer to \cite{main} for the continuous-time formulation.  The  turnpike property formalizes that an optimal solution of the {\em dynamic} optimal control problem is close to the  optimal solution of the corresponding {\em  static } problem. 
 
\begin{definition}\label{def:turnpike}
The optimal control problem $Q(t_{0},T,h,N,\psi^{0})$ 
defined in \eqref{5}  has the turnpike property with interior decay,  if there exist constants $\tilde{C}_{1}>0$,  $\lambda \in (0,1)$ and a monotone increasing non-negative function $\alpha$ with $\alpha(0)=0$,  such that for all $ T > t_0$ and $M$ given by equation \eqref{M} the following 
inequality holds true: 
\begin{equation}
\sum_{i=\left \lfloor (1-\lambda)M\right \rfloor}^{M-1}   \; h \; \alpha\left(  \Vert {\psi}^{i}-\psi^{(\sigma)}\Vert_N +\Vert {u}^{i}-u^{(\sigma)} \Vert_N   \right) \le \tilde{C}_1
\label{19}
\end{equation}
\end{definition}
Note that the constant $\tilde{C_1}$ can possibly depend on $h.$ This dependence will be studied in Section \ref{sec7}.

Our main result is the following theorem. It will be proven in the subsequent sections establishing first a dissipativity inequality and then showing that the problem is cheaply controllable. The final proof of the following theorem is deferred to Section~\ref{sec6}.

 \begin{theorem}\label{thm1}
Under assumption \eqref{assumption}, the optimal solution to $Q(t_{0},T,h,N,\psi^{0})$ has the turnpike property with interior decay in the sense of Definition \ref{def:turnpike}.
\end{theorem}

\section{The turnpike property with interior decay}             
\subsection{The strict dissipativity property}\label{sec3}
Dissipativity of the cost is crucial to obtain the turnpike property. The dissipativity has been introduced in \cite{measure} for the continuous problem \eqref{1}--\eqref{2}. Applied to the time discrete  agent system 
%
and optimal control problem $Q(t_{0},T,h,N,\psi^{0})$, it reads as follows:

\begin{definition}
The optimal control problem $Q(t_{0},T,h,N,\psi^{0})$ is called strictly dissipative  with respect to the supply rate function $\omega(\psi^{i},u^{i})$, if there exists a bounded  storage function $S: \mathbb{R}^{N\times d} \to \mathbb{R}$ and a monotone increasing continuous function $\alpha: \left [0,\infty\right) \to \left [0,\infty\right)$ with $\alpha(0)=0$ such that  for all $(\psi,u)$ fulfilling \eqref{3} holds
\begin{equation}
S(\psi^{i})+h \, \omega(\psi^{i},u^{i})  \ge S(\psi^{i+1}) +h \, \alpha  \left( \Vert \psi^i -\psi^{(\sigma)}\Vert_N + \Vert u^i-u^{(\sigma)}\Vert_N  \right).
\end{equation}
\end{definition}
\begin{lemma}\label{lemma1}
The  discrete time optimal control problem $Q(t_{0},T,h,N,\psi^{0})$ for $\gamma$  is strictly dissipative with $\omega(\psi^i,u^i) = g(\psi^i, u^i) -g(\psi^{(\sigma)}, u^{(\sigma)})$ where $g$ is as defined in (\ref{g}),  $S=0$ and $\alpha(x) = \frac\gamma{2 \; N} x^2.$
\end{lemma} 
\begin{proof}
Note that the optimal stationary state and control $(\psi^{(\sigma)}, u^{(\sigma)})$ fulfils $g(\psi^{(\sigma)}, u^{(\sigma)})=0$ due to equation \eqref{stationary solution}. By the definition of $\alpha$ and norm \eqref{norm} we have
\begin{align*}
	\alpha\left( \Vert \psi^i-\psi^{(\sigma)}\Vert_N + \Vert u^i -u^{(\sigma)}\Vert_N  \right) \leq  
	\frac{\gamma}{N} \left( \Vert \psi^i-\psi^{(\sigma)}\Vert^2_N + \Vert u^i-u^{(\sigma)}\Vert^2_N  \right) = \\
	\frac{\gamma}{N} \sum\limits_{k=1}^N 
	 \left( \Vert \psi^i_k-\psi_k^{(\sigma)}\Vert^2_N + \Vert u^i_k-u_k^{(\sigma)}\Vert^2 \right)  
	 	\leq 
	\frac{\gamma}{N} \sum\limits_{k=1}^N 
	\left( \Vert \psi^i_k- \bar{\psi}  \Vert^2_N + \Vert u^i_k \Vert^2 \right) = 
	g(\psi^i, u^i) - g(\psi^{(\sigma)}, u^{(\sigma)}).
\end{align*}
\end{proof}

\subsection{The cheap control condition}\label{sec4}

The turnpike property \ref{def:turnpike} asserts that the solution to the dynamic optimal control problem  is closed to the stationary optimal control problem for a large time interval.  The cheap control assumption defined below \eqref{definition:cheap}, introduces a bound depending on the  distance between the initial state and the steady state for the discrete in time problem. 
\begin{definition}\label{definition:cheap}
We say that  the optimal control problem \eqref{5} satisfies 
the cheap control condition 
if there exist constants $C_{0}>0$ and $\varepsilon_{0} \geq 0$ such that for all initial times $t_{0}$ and all initial states $\psi^{0}$ and for all terminal times $T>t_{0}$ the inequality
\begin{equation}
v_{N}^{h,t_{0},T}(\psi^{0})  \le C_{0}  \, \alpha(\Vert \psi^{0}-\psi^{(\sigma)}\Vert_N) + \varepsilon_{0} +S(\psi^M)-S(\psi^{0}).
\label{12}
\end{equation}
where $v_{N}^{h,t_{0},T}(\psi^{0})$ is the optimal value of the control problem \eqref{v8}  and $\psi^M$ is the optimal state at time $T$  given by equation \eqref{3} for $M$ given by equation \eqref{M}.
\end{definition}
Also, here the constant $C_0$ can depend on $h$ and its dependence is discussed in Section \ref{sec7}.

\begin{lemma}\label{lemma2}
Assume that $P$ fulfils \eqref{assumption}.  The optimal control problem $Q(t_{0},T,h,N,\psi^{0})$ is cheaply controllable in the sense of Definition \ref{definition:cheap} for $\varepsilon_{0}=0$,   $\alpha(x)=\frac{\gamma}{2N}x^2$ and $S=0$ as in Lemma \eqref{lemma1}.
\end{lemma}
\begin{proof}
Due to the choice of $\alpha,S$ and $\varepsilon_{0}$ we only need to show, that there  exists a constant $\tilde{C_{0}}$ such that for all initial times $t_{0}$ and for all terminal times $T>t_{0}$ and all initial states $\psi^{0} \in \mathbb{R}^{N \times d}$ we have the inequality
\begin{equation}
	v_{N}^{h,t_{0},T}(\psi^{0}) \le \tilde{C_{0}}  \, \frac{\gamma}{2}\frac{1}{N} \left(\Vert \psi^{0}-\psi^{(\sigma)}\Vert_{N} \right)^{2}.
	\label{13}
\end{equation}
To prove inequality \eqref{13}   we  consider the same stabilizing feedback law for controls in \cite{main} that leads to exponential decay of $g(\psi^i, u^i)$ given by equation \eqref{g}. Let $h=\frac{T-t_{0}}{M} $ and  let $\frac{1}h \geq \beta>0$ be given and define for all
$k=1,..,N,$ and $i=0,...,M$ 
\begin{equation}
	u_{k}^{i}=\beta(\bar\psi-\psi_{k}^{i})-\frac{1}{N}{\sum_{l=1}^{N}P(\psi_{k}^{i},\psi_{l}^{i}) (\psi_{l}^{i}-\psi_{k}^{i})}.
	\label{14}
\end{equation}
The corresponding state $\psi^i$ for any initial state $\psi^0$ at time $t_0$ is then given by the dynamics \eqref{3}, i.e.,   
\begin{equation}
	\psi_{k}^{i+1}=\psi_{k}^{i}+ h \, \beta \; \left(\bar\psi-\psi_{k}^{i} \right). 
	\label{15}
\end{equation}
This implies 
\begin{equation*}
	\psi_{k}^{i}=(1-h\beta)^{i}\psi_{k}^{0}+\sum_{l=0}^{i-1}(1-h\beta)^{l}h\beta \; \bar\psi.
\end{equation*}
Since by  definition of $\beta$ we have  $h \; \beta<1$, we obtain 
\begin{align*}
\Vert\psi_{k}^{i}-\bar\psi\Vert =\Vert(1-h\beta)^{i}\psi_{k}^{0}+[\sum_{l=0}^{i-1}(1-h\beta)^{l}h\beta-1]\bar\psi \Vert =(1-h\beta)^{i}\Vert\psi_{k}^{0}-\bar\psi\Vert.
\end{align*}
The control $u^i_k$ is estimated by 
\begin{align*}
	\Vert u_{k}^{i}\Vert &  \le \beta \Vert \bar\psi-\psi_{k}^{i}\Vert+\frac{\Vert P\Vert}{N} \sum_{\ell=1}^{N}{\Vert\psi_{k}^{i}-\bar\psi\Vert+\Vert\psi_{\ell}^{i}-\bar\psi\Vert} 
	= (\beta+\Vert P\Vert)\Vert\bar\psi-\psi_{k}^{i}\Vert+\frac{\Vert P\Vert}{N} \sum_{\ell=1}^{N}{\Vert\psi_{\ell}^{i}-\bar\psi\Vert} \\
	& 
 \le (1-h\beta)^{i}  \left( \left(\beta+\Vert P\Vert\right)\Vert\psi_{k}^{0}-\bar\psi\Vert+\frac{\Vert P\Vert}{N} \sum_{\ell=1}^{N}{\Vert\psi_{\ell}^{0}-\bar\psi\Vert} \right). 
 \end{align*}
 Therefore, 
 \begin{align*}
\Vert u^i_k \Vert^2 & \le 2 (1-h\beta)^{2i}   \left( \left(\beta+\Vert P\Vert\right)^2 \Vert\psi_{k}^{0}-\bar\psi\Vert^2 +\frac{\Vert P\Vert^2}{N} \sum_{\ell=1}^{N}{\Vert\psi_{\ell}^{0}-\bar\psi\Vert^2 } \right), 
\end{align*}
and 
 \begin{align*}
	\Vert u^i \Vert_N^2 & \le 2 (1-h\beta)^{2i}   \left( \left(\beta+\Vert P\Vert\right)^2 + \Vert P \Vert^2 \right)  \Vert\psi_{\ell}^{0}-\bar\psi\Vert^2_N. 
\end{align*}
Also, the state at time $i$ are bounded by the  initial state 
\begin{equation*}
 \frac1N \Vert \psi^i - \bar\psi \Vert_N^2 = \frac{1}{N}\sum_{k=1}^{N}\Vert\psi_{k}^{i}-\bar\psi\Vert^{2} = 
 \frac{ (1-h\beta)^{2i} }N \Vert \psi^0 - \bar{\psi} \Vert^2_N.
\end{equation*}
Combining the previous estimates yields for the cost $J_{N}^{h,t_{0},T}(\psi, u)$ the inequality
\begin{align*}
	J_{N}^{h,t_{0},T}(\psi, u) &\le \frac{h}N  \sum_{i=0}^{M-1}(1-h\beta)^{2i} \left( 1+2\gamma((\beta+\Vert P \Vert)^{2}+\Vert P\Vert^{2}) \right) \Vert \psi^{0}-\bar\psi\Vert_{N}^{2}\\
	&=\frac{h-h(1-h\beta)^{2M}}{1-(1-h\beta)^{2}} \left( 1+2\gamma(\beta^{2}+2\beta\Vert P \Vert+2\Vert P\Vert^{2})\right) \frac{1}{N}\Vert \psi^{0}-\bar\psi\Vert_{N}^{2}.
\end{align*}
Therefore, the optimal value \eqref{v8} of problem \eqref{5} fulfils
\begin{align*}
	v_{N}^{h,t_{0},T}(\psi^{0})\le\frac{h}{1-(1-h\beta)^{2}} \left( 1+2\gamma(\beta^{2}+2\beta\Vert P \Vert+2\Vert P\Vert^{2})\right) \frac{1}{N}\Vert \psi^{0}-\bar\psi\Vert_{N}^{2}.
\end{align*}
Hence, equations \eqref{12} and  \eqref{13} hold true with 
\begin{equation}
	\label{c0gleichung}
	\tilde C_{0} :=\frac{h}{1-(1-h\beta)^{2}} \left( 
	\frac{2}{\gamma}+4(\beta^{2}+2\beta\Vert P \Vert+2\Vert P\Vert^{2}) \right).
\end{equation}
\end{proof}


\subsection{Proof of Theorem \eqref{thm1}}\label{sec6}

Next, we turn to the proof of Theorem \ref{thm1} based on the previous lemmas on strict dissipativity and the cheap control property of the time discrete system. 

\begin{proof}
Under the assumptions the problem \eqref{5} fulfils the strict dissipativity inequality \eqref{10} and it is cheaply controllable \eqref{12} for $S=\varepsilon_{0}=0$ and $\alpha(x) = \frac{\gamma}{2 N}x^2.$
Denote by  $({\psi}^*, \hat{u}^*)$ the optimal state and control to problem \eqref{5}. Then, the strict dissipation inequality   and  the cheap control yields 
\begin{equation}
\begin{aligned}
h\sum_{i=0}^{M-1} \alpha\left(\Vert {\psi}^{*,i}-\psi^{(\sigma)}\Vert_{N}+\Vert {u}^{*,i}-u^{(\sigma)}\Vert_{N}\right) \le h\sum_{i=0}^{M-1}g({\psi^{*,i}}, {u^{*,i}})  \\ =v_{N}^{h,t_{0},T}(\psi^{0})
 \le \tilde{C_{0}}  \, \alpha(\Vert \psi^{0}-\psi^{(\sigma)}\Vert_{N}).
\label{20}
\end{aligned}
\end{equation}
Let \begin{equation}
    \label{r1}
    r_{1}=\left \lfloor (1-\lambda)M\right \rfloor. \end{equation}

Suppose  that for all $i \in \left\{0,1,\dots, r_{1}-1\right\}$ the inequality $$h\alpha(\Vert {\psi}^{*,i}-\psi^{(\sigma)}\Vert_{N}+\Vert {u}^{*,i}-u^{(\sigma)}\Vert_{N})>\frac{1}{r_{1}} \tilde{C_{0}}  \, \alpha(\Vert \psi^{0}-\psi^{(\sigma)}\Vert_{N})$$holds. Then, this implies $$h\sum_{i=0}^{r_{1}-1} \alpha(\Vert {\psi}^{*,i}-\psi^{(\sigma)}\Vert_{N}+\Vert {u}^{*,i}-u^{(\sigma)}\Vert_{N}) > \tilde{C_{0}} \, \alpha(\Vert \psi^{0}-\psi^{(\sigma)}\Vert_{N})$$ and therefore  a contradiction to inequality  \eqref{20}. Hence, 
there exists an index $r_{1}^{*} \in  \left\{0,1,\dots, r_{1}-1\right\} $ such that
\begin{equation}
h\; \alpha(\Vert {\psi}^{*, r_{1}^{*}}-\psi^{(\sigma)}\Vert_{N}+\Vert {u}^{*,r_{1}^{*}}-u^{(\sigma)}\Vert_{N}) \le \frac{1}{r_{1}} \tilde{C_{0}} \, \alpha(\Vert \psi^{0}-\psi^{(\sigma)}\Vert_{N}).
\label{21}
\end{equation}
This implies 
\begin{equation}
\begin{aligned}
\sum_{i=r_{1}}^{M-1}h\alpha(\Vert{\psi}^{*,i}-\psi^{(\sigma)}\Vert_{N}+\Vert {u}^{*,i}-u^{(\sigma)}\Vert_{N}) &\le
\sum_{i=r_{1}^{*}}^{M-1}h\alpha(\Vert{\psi}^{*,i}-\psi^{(\sigma)}\Vert_{N}+\Vert {u}^{*,i}-u^{(\sigma)}\Vert_{N}) \\& \le 
\sum_{i=r_{1}^{*}}^{M-1}hg({\psi^{*,i}}, {u^{*,i}})  =v_{N}^{h,t_{r_{1}^{*}},T}({\psi}^{*, r_{1}^{*}}).
\end{aligned}
\label{22}
\end{equation}
The  value  $v_{N}^{h,t_{r_{1}^{*}},T}({\psi}^{*, r_{1}^{*}})$ is the optimal value of the optimization problem $Q(t_{r_{1}^{*}},T,h,N,{\psi}^{*, r_{1}^{*}})$ with initial time $t_{r^{*}}$ and initial state ${\psi}^{*, r_{1}^{*}}$. The cheap control results of Lemma \ref{lemma2} yields 
\begin{equation}
v_{N}^{h,t_{r_{1}^{*}},T}({\psi}^{*,r_{1}^{*}}) \le \tilde{C_{0}} \, \alpha\left( \Vert {\psi}^{*, r_{1}^{*}}-\psi^{(\sigma)}\Vert_{N} \right)).
\label{23}
\end{equation}
Together with equation \eqref{21} this yields 
\begin{equation}
v_{N}^{h,t_{r_{1}^{*}},T}({\psi}^{*, r_{1}^{*}}) \le \frac{\tilde{C_{0}}^{2}}{hr_{1}} \, \alpha(\Vert \psi^{0}-\psi^{(\sigma)}\Vert_{N}), 
\end{equation}
and combining  with equation \eqref{22} we establish the estimate 
$$ \sum_{i=r_{1}}^{M-1}h\alpha(\Vert{\psi}^{*,i}-\psi^{(\sigma)}\Vert_{N}+\Vert {u}^{*,i}-u^{(\sigma)}\Vert_{N}) \le  \frac{\tilde{C_{0}}^{2}}{h \; r_{1}} \alpha(\Vert \psi^{0}-\psi^{(\sigma)}\Vert_{N}).$$
We define the $h-$dependent constant $\tilde{C}_1=\tilde{C}_1(h)$ by 
\begin{equation}\label{c1h}
    \tilde{C}_1:=  \frac{\tilde{C_{0}}^{2}}{h \; r_{1}}.
\end{equation}
Since $r_1$ is given by equation \eqref{r1}, the product $h r_1$ is bounded from below by $(1-\lambda)T.$
\end{proof}

In the comparisons of the numerical simulations later on, the following result can be established. 

\begin{lemma}
Assume $\beta>0$ such that $h \beta <1.$ 
For the cheap control \eqref{14}, the discrete Lyapunov function for $i=0,\dots,$
\begin{align}
    L^i_N := \frac1N \Vert \psi^i - \bar\psi \Vert_N^2
\end{align}
along the solution $\psi^i$ given by \eqref{3} 
decays at each discrete time step  with the rate
\begin{align}
    r = \left( 1 - h \beta \right)^{2}
\end{align}
\end{lemma}
Due to the relation \eqref{14}, we have that 
\begin{align}
    L^i_N = \frac{1}N ( 1 - h \beta)^{2i} \Vert \psi^0 - \bar\psi \Vert^2_N = (1-h \beta)^{2 i } L^0_N.
\end{align}


\section{Uniformity of the turnpike property
for $h\rightarrow 0^+$}\label{sec7}

In this section, we show that the turnpike property 
for each discretization level $h>0$, 
also holds uniformly for  $h\rightarrow 0^+$.
This completes the picture about the turnpike property 
that we have established already for the continuous  in time problem in \cite{main}. 
The turnpike property possesses a continuous transition between the time-discrete level and the continuous description.
\par 
Due to (\ref{c0gleichung} and for a given stepsize $h$,  the constant in the 
inequality 
(\ref{12})
in the cheap control condition
is given by 
\begin{align*}
\tilde C_{0}(h)    
&=
\frac{1}{2 \, \beta - h \,\beta^2}
\left[\frac{2}{\gamma}+4(\beta^{2}+2\beta\Vert P \Vert+2\Vert P\Vert^{2})\right].
\end{align*}
This implies
\[
\lim_{h \rightarrow 0^+}
\tilde C_0(h)
=
\frac{1}{\beta}
\left[\frac{1}{\gamma}+ 2(\beta^{2}+ 2\beta\Vert P \Vert+ 2 \Vert P\Vert^{2})\right] =: \tilde D_0
\]
and
$\tilde C_0(\cdot)$ is strictly decreasing with respect to $h$.
Hence, for all $h \in (0, \, \frac{1}{\beta})$
we have the inequality 
\[\tilde D_0 < \tilde C_0(h) \leq  2 \tilde D_0.\]
Thus  inequality 
(\ref{12})
in the cheap control condition
with the constant $  2 \tilde D_0 $
holds  uniformly with respect to $h$.

Furthermore, we have a similar estimate for $\tilde{C}_1$ given by equation \eqref{c1h}. 
Note that denominator $h r_1$ is bounded from below by $(1-\lambda)T$ and the numerator 
$\tilde{C}_0^2$ is bounded from above for  $h \in (0, \, \frac{1}{\beta})$. Hence, $\tilde{C}_1$ 
is bounded independently of $h$ and therefore the turnpike inequality \eqref{19} holds uniformly for
all $h\in (0, \, \frac{1}{\beta}).$


\section{Numerical experiments}\label{sec8}

We provide several examples in order to illustrate the turnpike phenomenon for optimal control problems.  According to the analysis in this article, the turnpike property with interior decay gives an intuition that for a sufficiently large time horizon $T-t_{0}$, approximations of the optimal control-state pairs should be close to the optimal steady state after the middle of the interval. For a given initial state and discrete scheme of the systems governed by ordinary differential equations, we can get the state $\psi_{k}$ and control $u_{k}$ of each particle by iteration at the time points $t_{i}$, we implement this process in Matlab. To solve the optimal problem with unknowns $\psi_{k}^{i}, \, u_{k}^{i}, \, i=1,\dots,M, \, k=1,\dots,N$, we use the optimization routine IPOPT (see \cite{IPOPT}) combined with AMPL, and plot the results in Matlab. 
We randomly generate the set of initial state $\psi_{k}^{0}$ uniformly in [0,1] at initial time $t_{0}=0$ for $N=100$ agents moving in $d=1$ dimensional space up to a final time $T=5$. Setting the desired consensus point $\bar{\psi}=0.5$ and the penalization for the control energy $\gamma=0.1$, we implement also the cheap control \eqref{14} presented in \ref{sec4}. 
Test 3 also  contains a comparison for increasing numbers of agents $N\in \{50, \, 100, \, 500\}$.

\paragraph{Test 1: Trajectories and controls, the turnpike phenomenon.}

First we consider $P(x,y)=\Vert x-y \Vert ^{2}$ as the interaction kernel, where $x,y\in \mathbb{R}^{d}$. The step-size is set to $h=\frac{T-t_{0}}{M}=0.01$, hence $M=500$ time-steps.
The uncontrolled case, that is for $u = 0$, is shown in Figure \ref{fig:T1_a}, we can see that the mean is conserved and the dynamics converge towards its stable equilibrium position which may different from the prescribed consensus state.

The cheap controlled with different parameter $\beta$ and optimal controlled results are presented in Figure \ref{fig:T1_b}. It is observed that in all the cases, the agents tends to the prescribed consensus point.
Although by setting larger value for $\beta$ the states and controls can arrive to the steady point earlier, it may take more total running cost, and the control pays more effort at the beginning period of the time horizon.

In Figure \ref{fig:T1_c}, the decays of $L_{N}$ (top-left) and of the running cost $g$ (top-right) are presented, for the five different cases discussed above.
For a larger value of $\beta$, we can numerically confirm that  $L_{N}$ decays faster according to the theoretical analysis of it. However, as shown on the top-right plot, it takes more total running cost. We can also observe that for $\beta=3$ the cheap control is very close to the optimal one, and it can be considered a suitable approximation of it. 
 In the lower plots of Figure \ref{fig:T1_c} we actually show the turnpike phenomenon, on the left we depict the gap between the optimal states of all particles $\sum_{k=1}^{N}\psi^{i}_{k}$  and their steady states $\sum_{k=1}^{N}\bar{\psi}$, on the right the difference between the optimal control and static control $u=0$. We observe that after an initial stabilization phase, the dynamic optimal state and control are close to the static ones.

\begin{figure}[t]
\hspace{4.5cm}\includegraphics[width=0.8\linewidth]{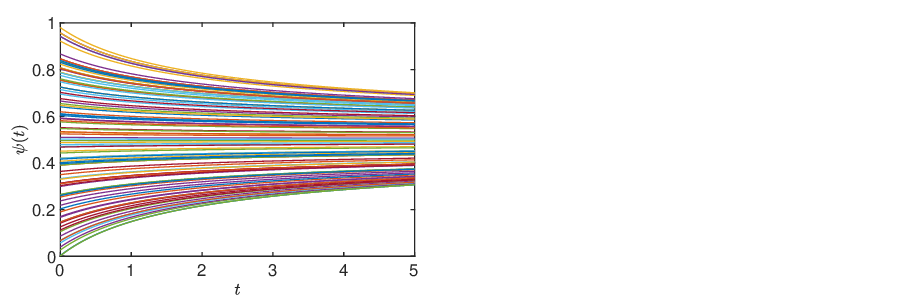}
\caption{\textbf{Test 1.} Trajectories evolution in time for $N=100$ agents without control. The uncontrolled dynamics converge to its equilibrium point.}
\label{fig:T1_a}
\end{figure}
\begin{figure}[t]
\begin{center}
\includegraphics[width=0.24\textwidth]{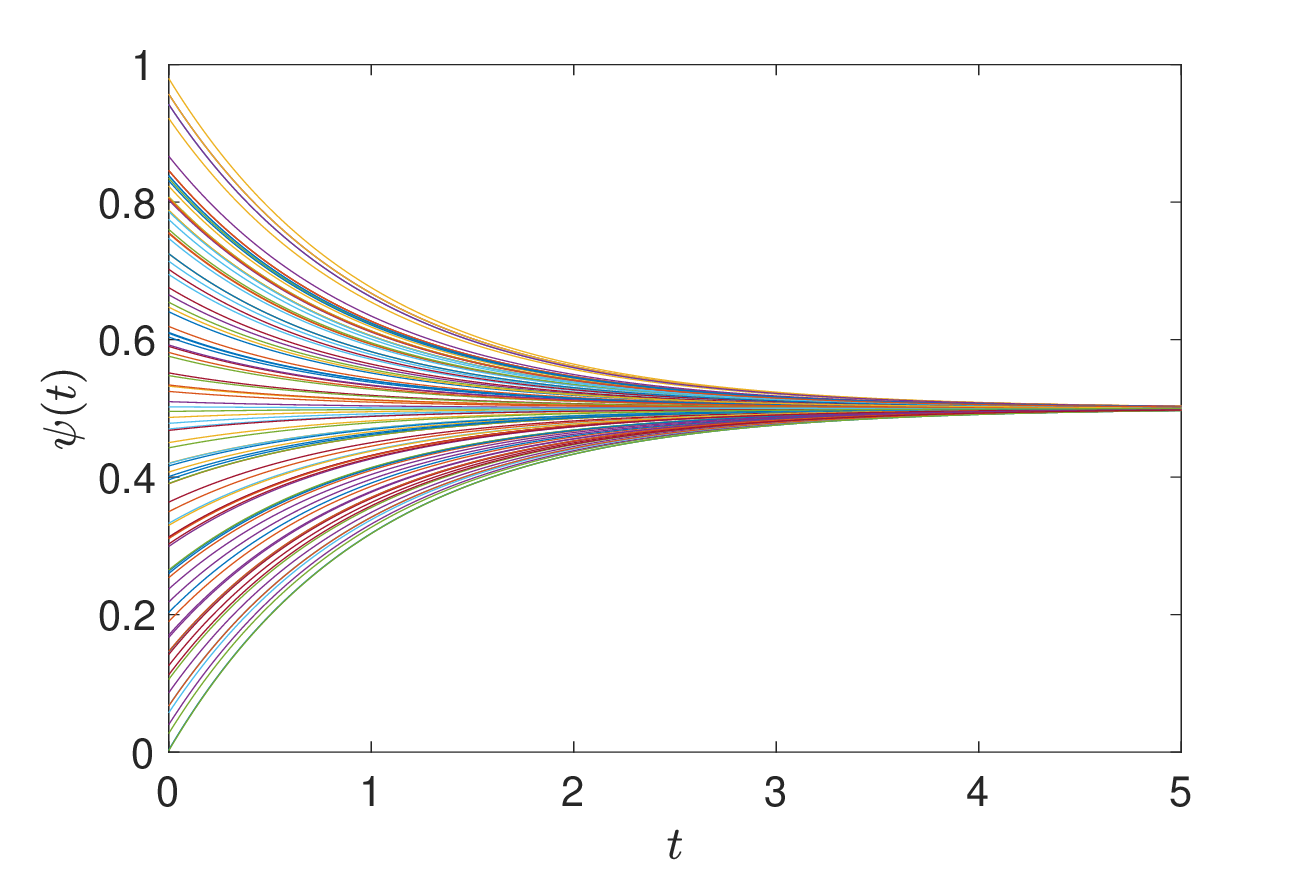}
\includegraphics[width=0.24\textwidth]{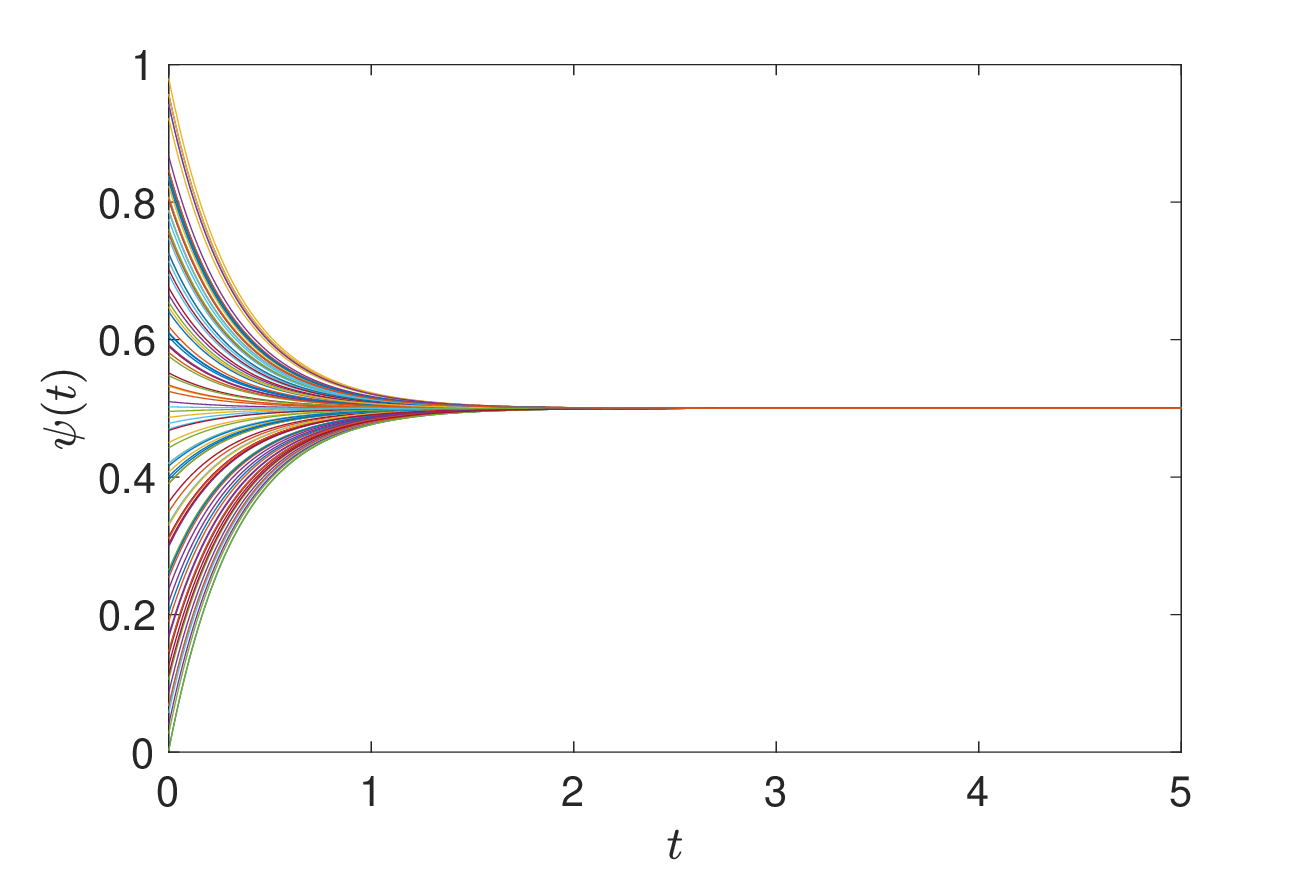}
\includegraphics[width=0.24\textwidth]{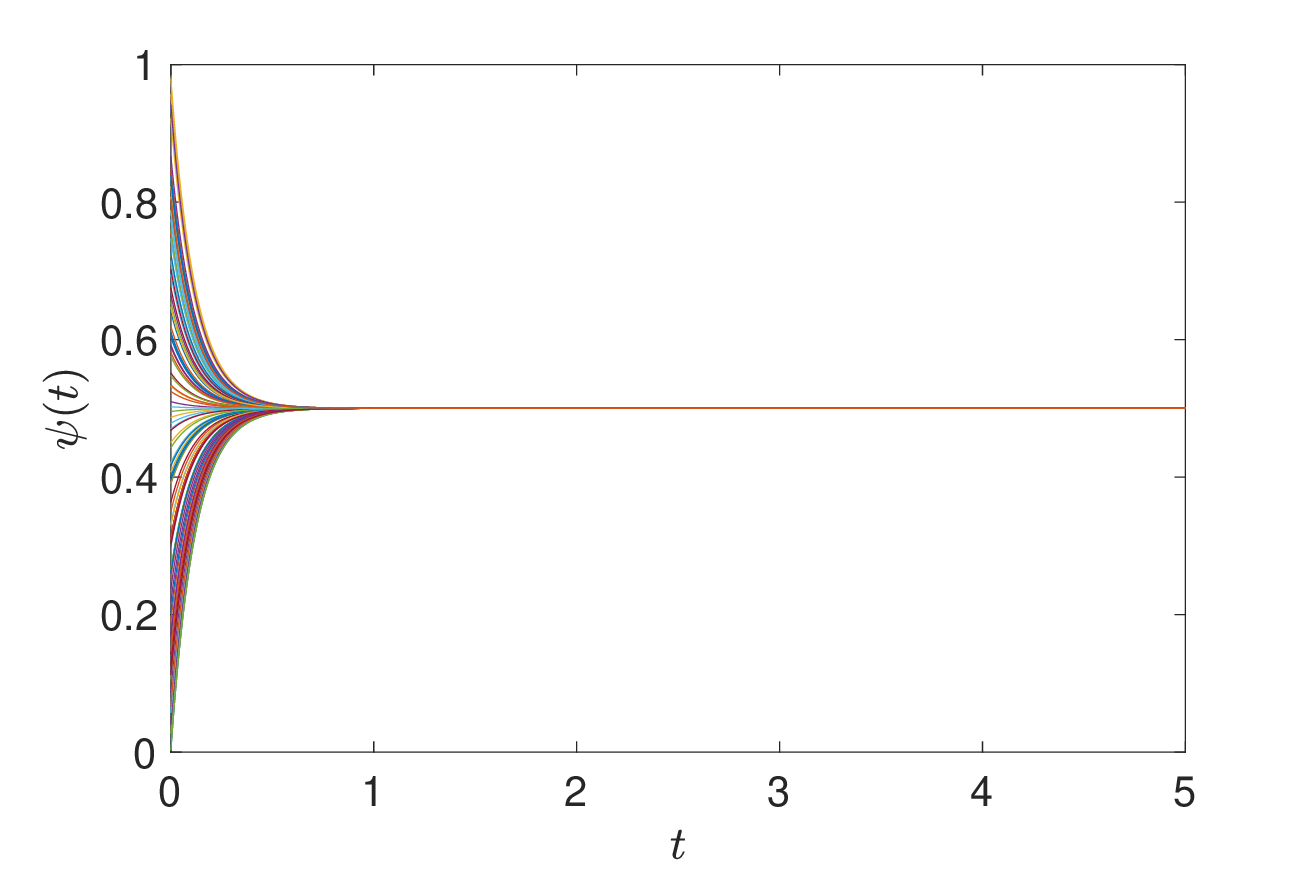}
\includegraphics[width=0.24\textwidth]{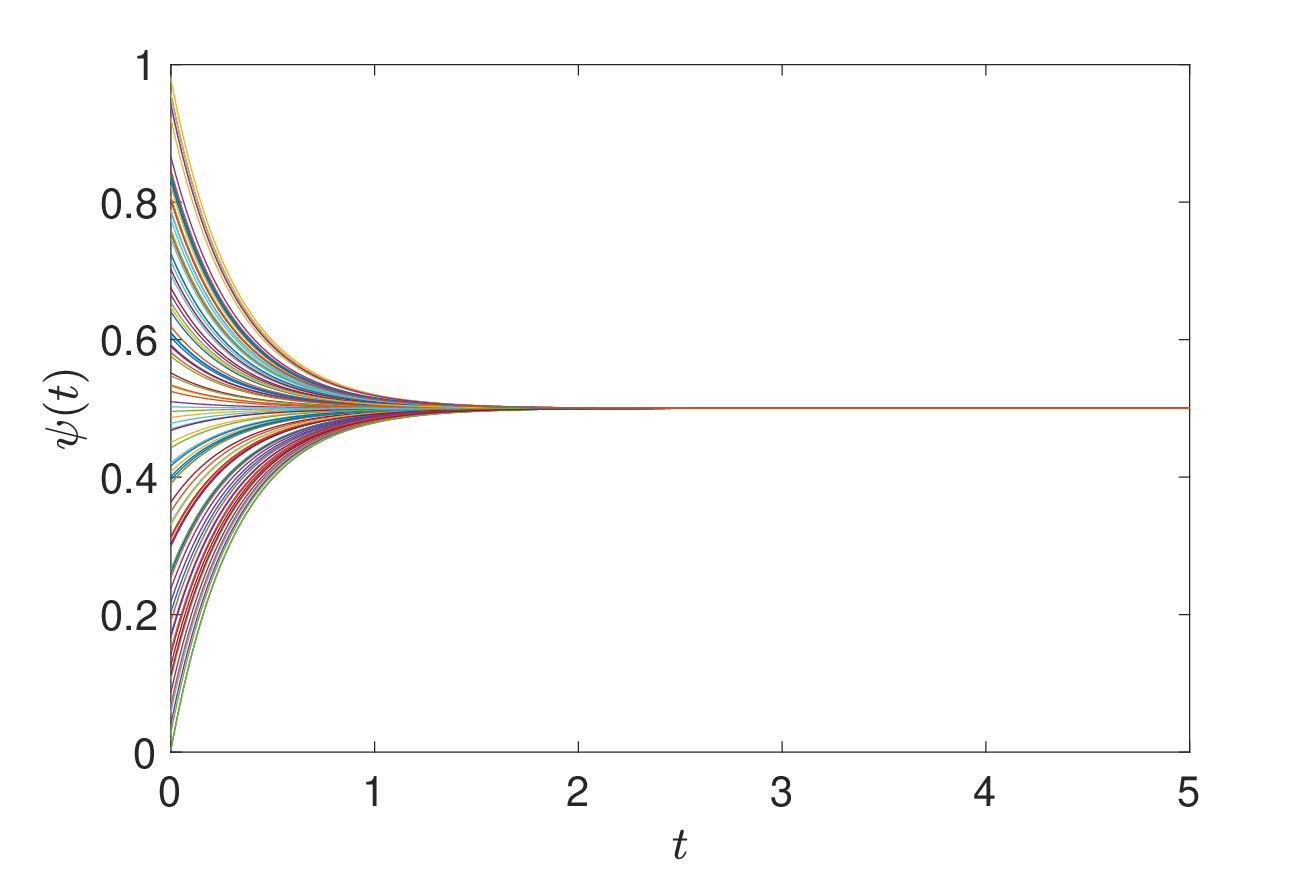}
\\
\includegraphics[width=0.24\textwidth]{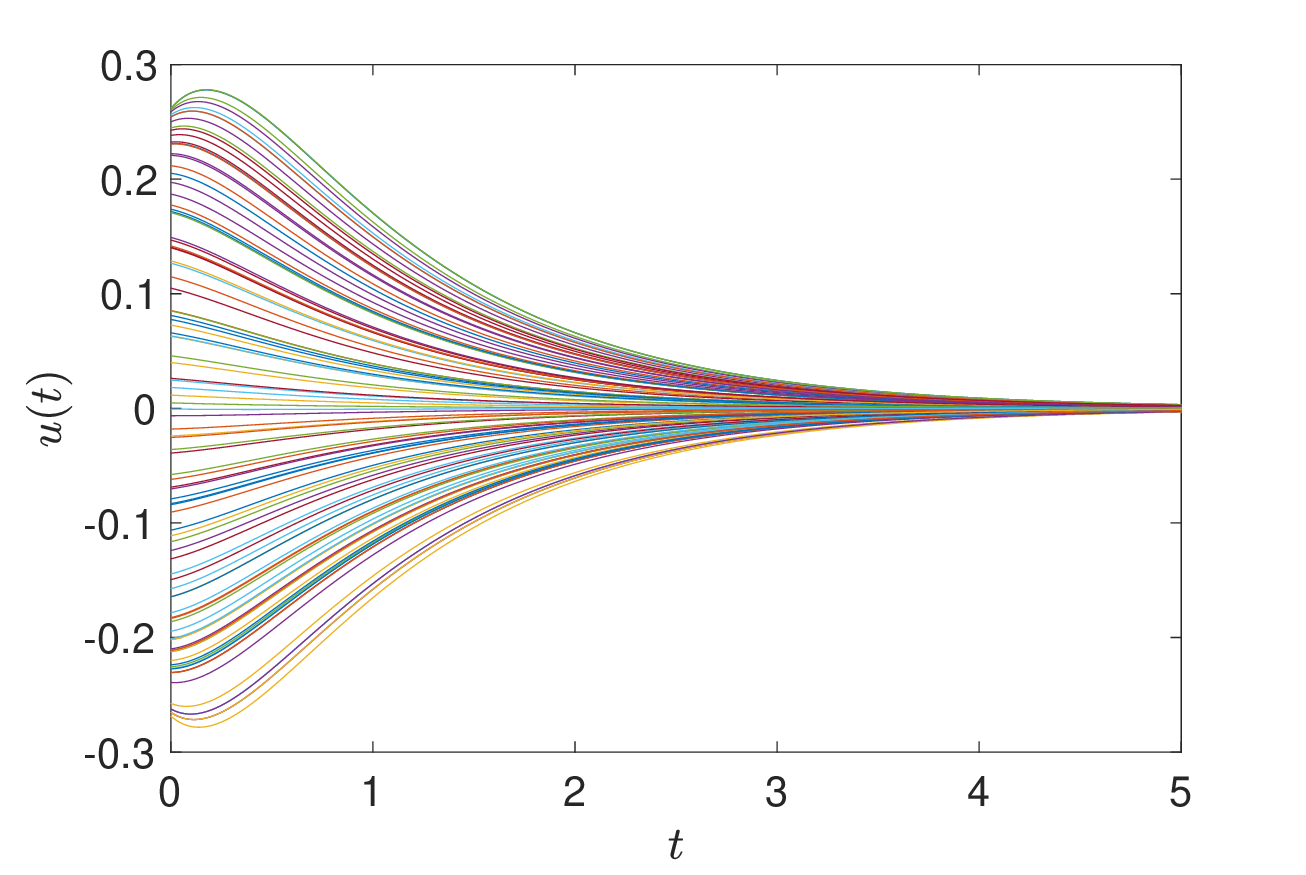}
\includegraphics[width=0.24\textwidth]{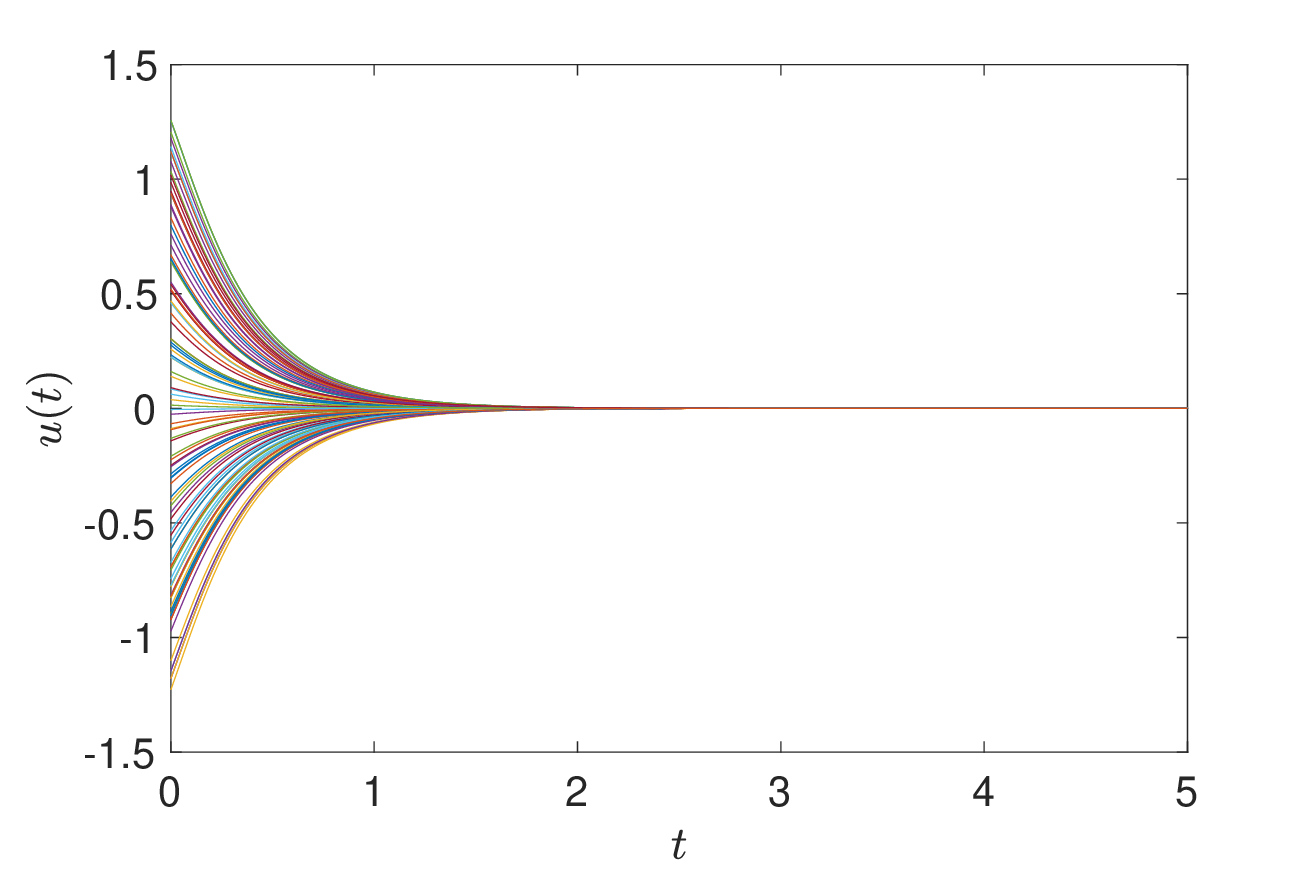}
\includegraphics[width=0.24\textwidth]{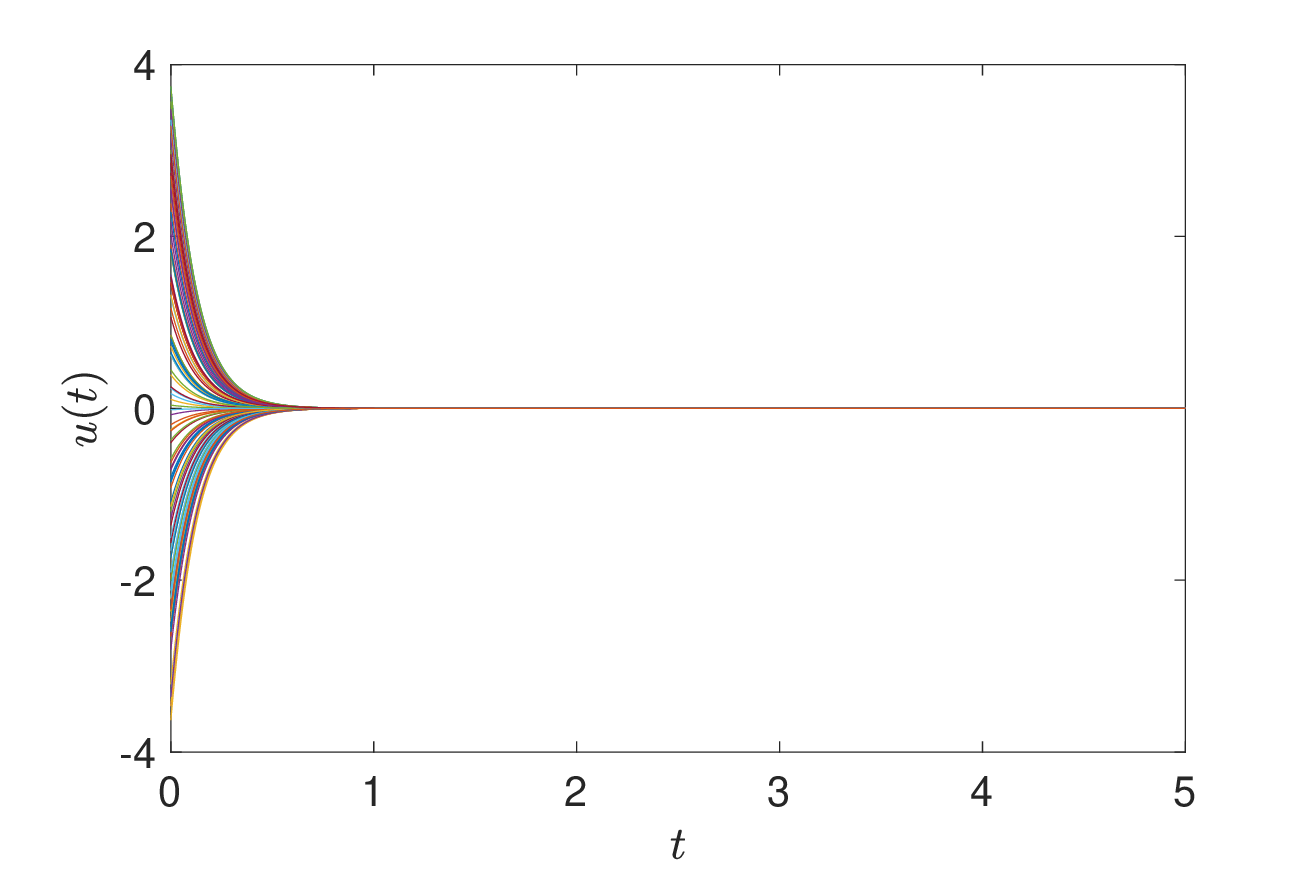}
\includegraphics[width=0.24\textwidth]{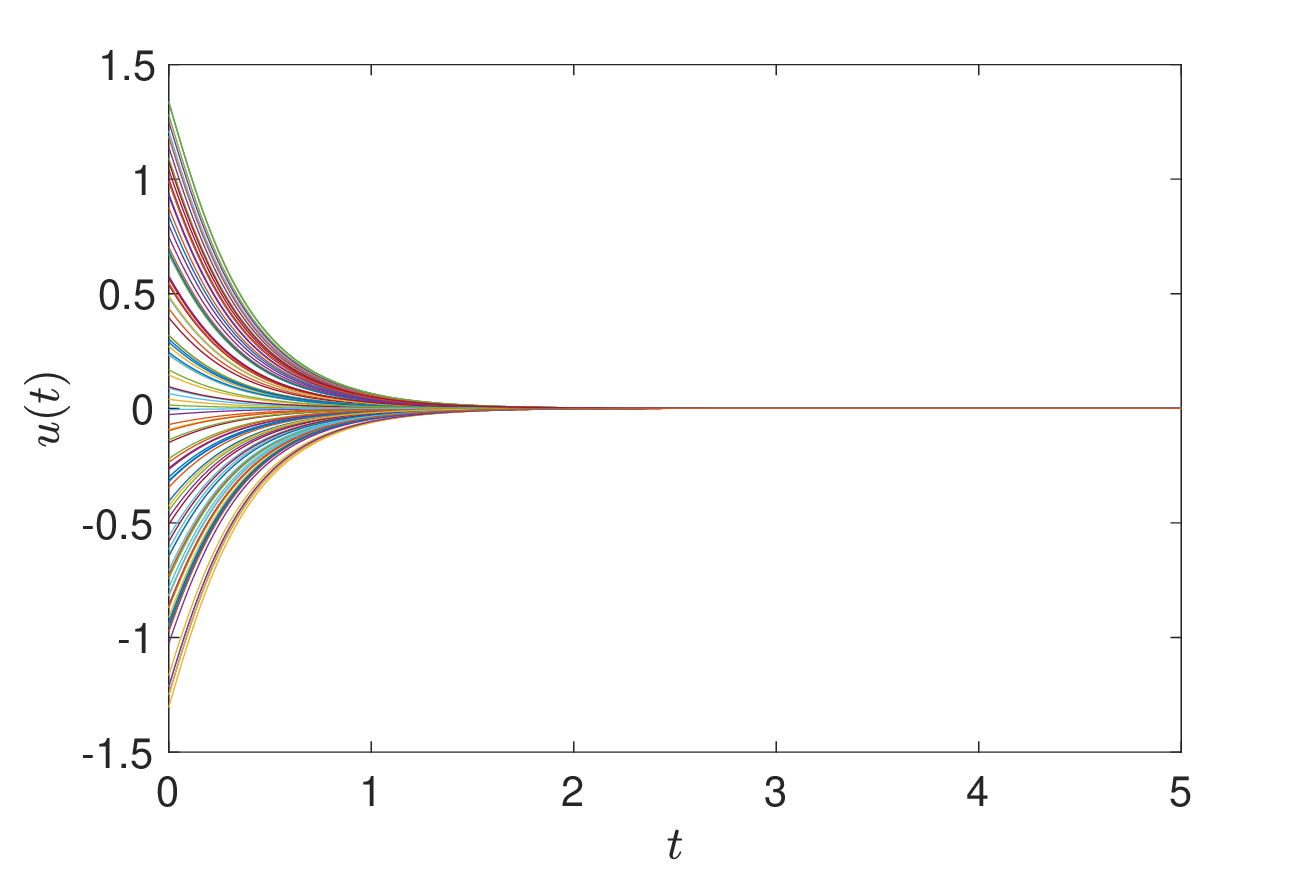}
\caption{\textbf{Test 1.} From left to right, evolution of states (top) and controls (bottom) for $N=100$ agents which are steered by cheap control with $\beta=1,3,8$ and optimal control respectively.
The larger the value of $\beta$, the faster the agents reach the desired state.}
\label{fig:T1_b}
\end{center}
\end{figure}
\begin{figure}[t]
\begin{center}
\includegraphics[width=0.4\textwidth]{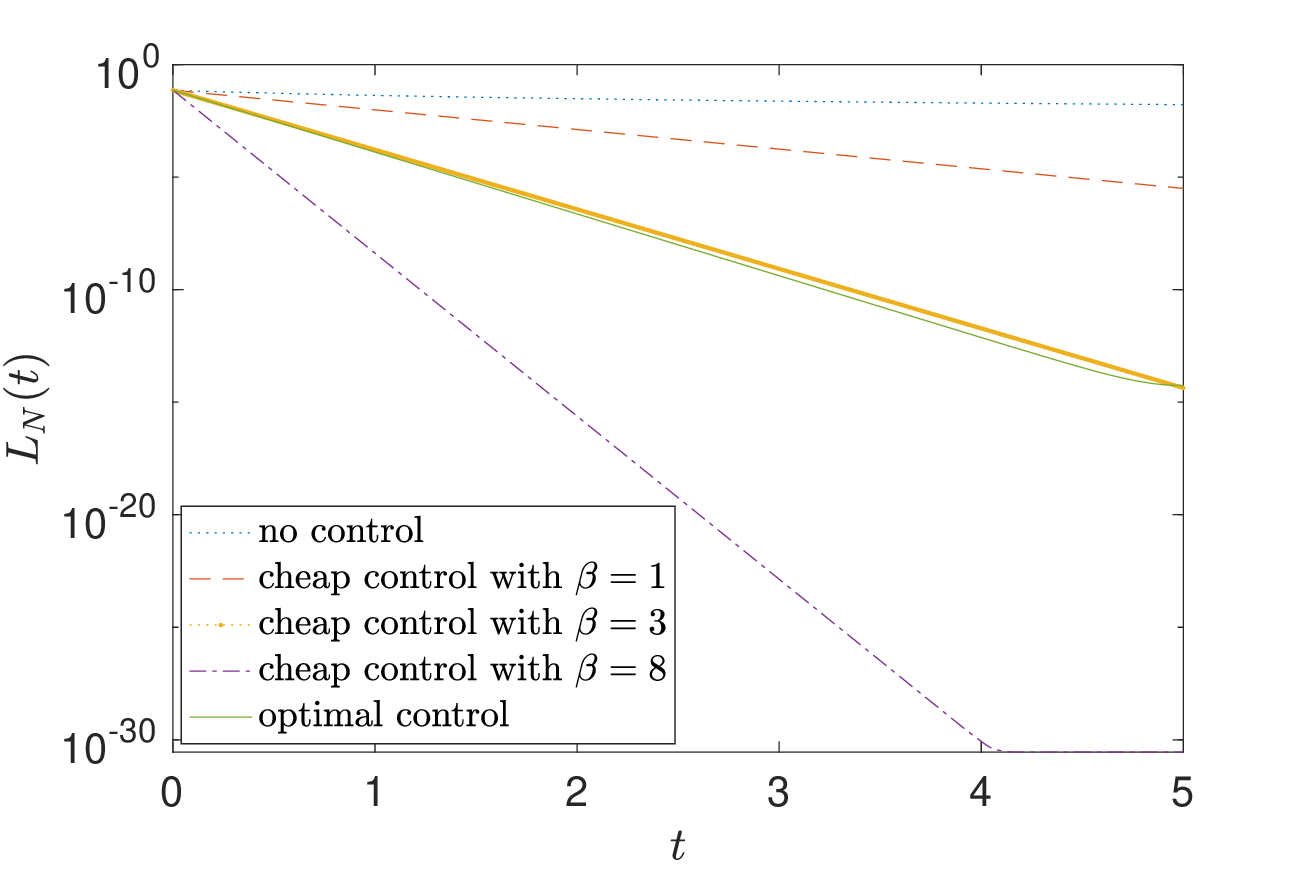}
\includegraphics[width=0.4\textwidth]{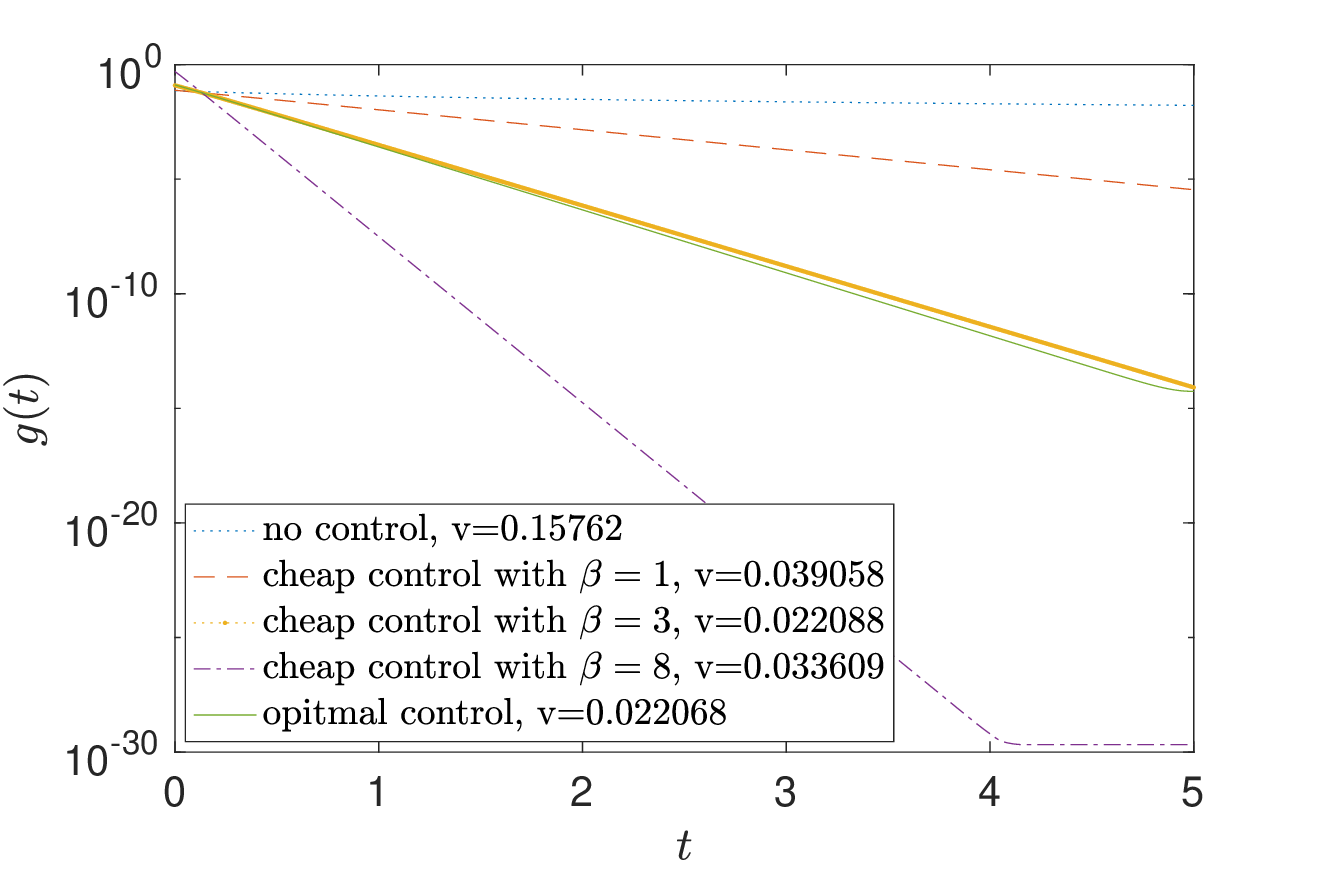}
\\
\includegraphics[width=0.8\textwidth]{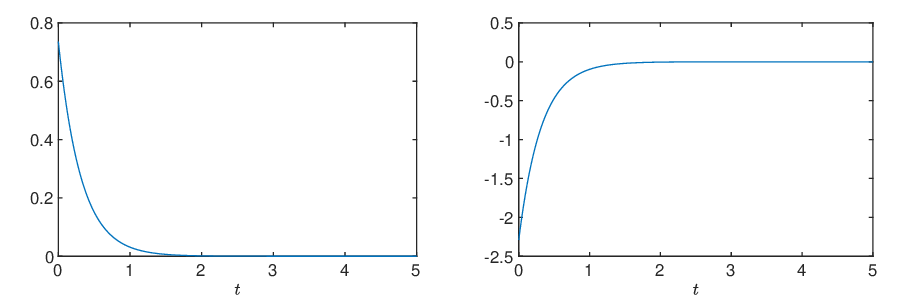}
\caption{\textbf{Test 1.} On the top, the Lyapunov functions on a logarithmic scale (left), and the running cost over time (right). On the bottom, the turnpike property with interior decay, i.e. the gap between the optimal and static states of all particles
$\sum_{k=1}^{N}\psi^{i}_{k} - \sum_{k=1}^{N}\bar{\psi} $ (left) and the gap between optimal and static control
$\sum_{k=1}^{N}u^{i}_{k}$ (right) over time.}
\label{fig:T1_c}
\end{center}
\end{figure}

\paragraph{Test 2: Comparison for different time-steps.} \label{par:test2}

Our numerical study delves into an exploration of the turnpike property, observing its behavior across various discretization levels. Since it is a close approximation of the optimal solution, we implement the cheap control with $\beta=3$. In Figure \ref{fig:T2_a}, we show the trajectories of both state and control variables, providing a representation of their evolution under different time-discretization steps $h=0.1, , 0.01, , 0.001$.

Furthermore, Figure \ref{fig:T2_b} depicts the logarithmic scale representation of the exponential decay of $L_N$ for each discretization level. This offers valuable insights into the system's behavior, providing a more refined understanding of the turnpike property and the efficacy of our implemented control strategy.

\begin{figure}[t]
\begin{center}
{\includegraphics[width=0.24\textwidth]{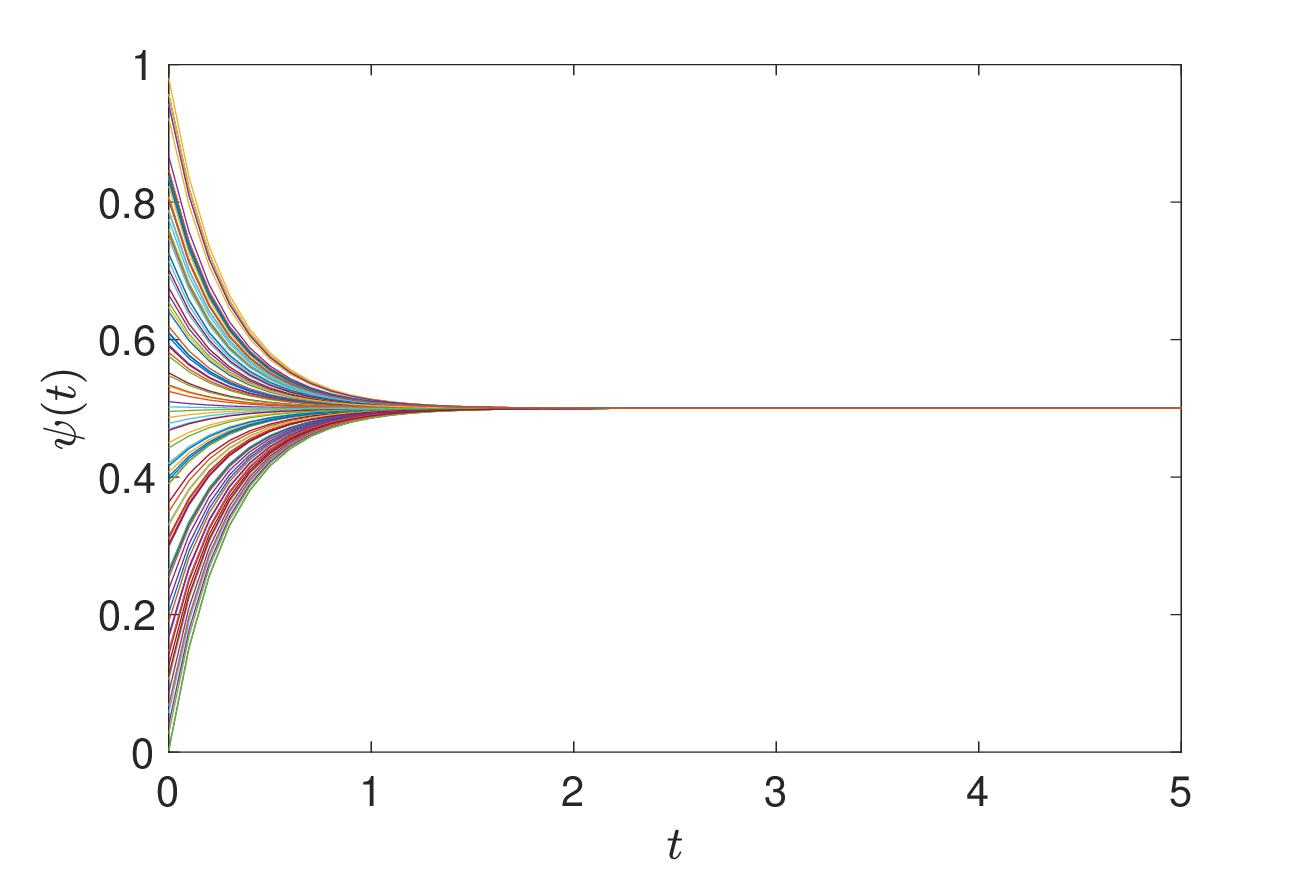}}
{\includegraphics[width=0.24\textwidth]{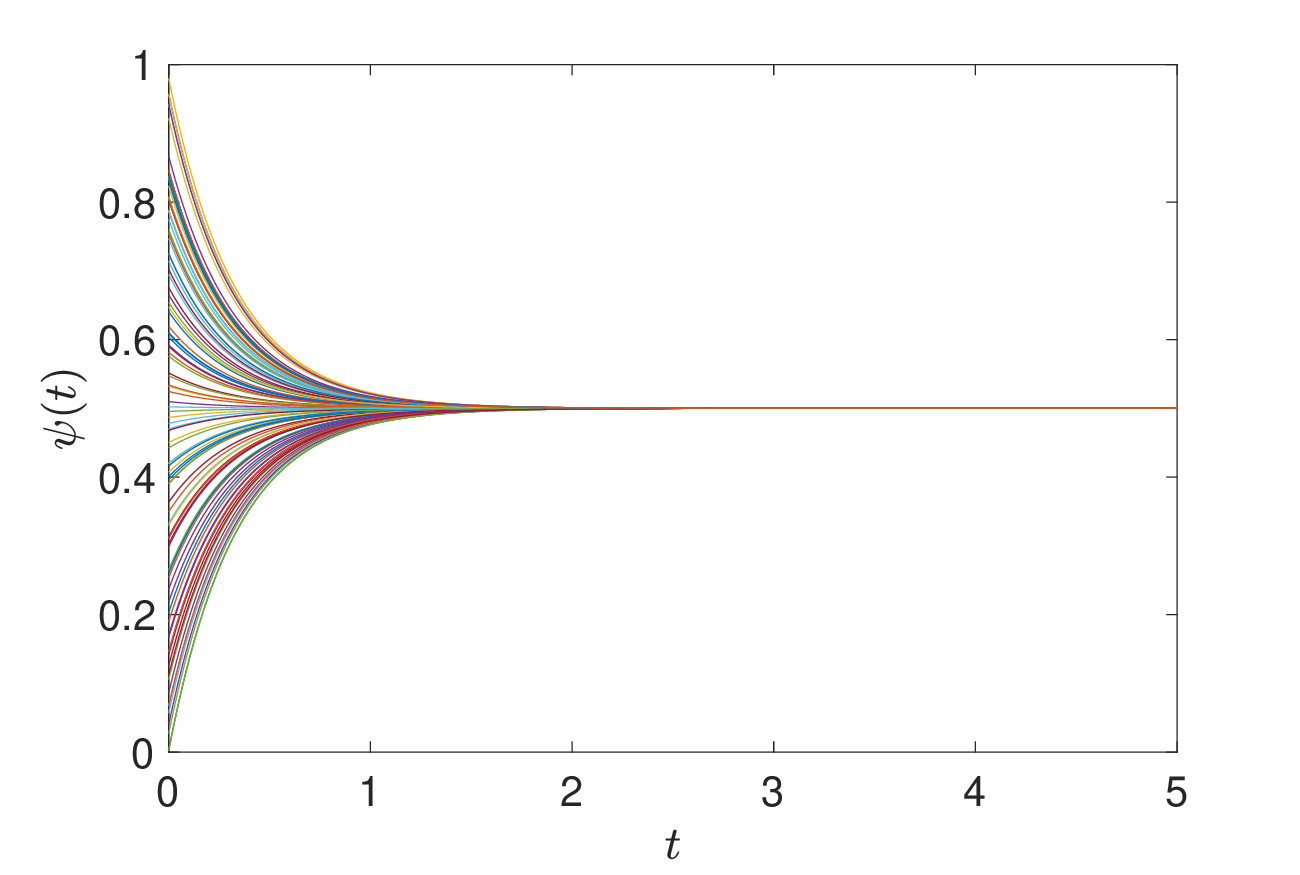}}
{\includegraphics[width=0.24\textwidth]{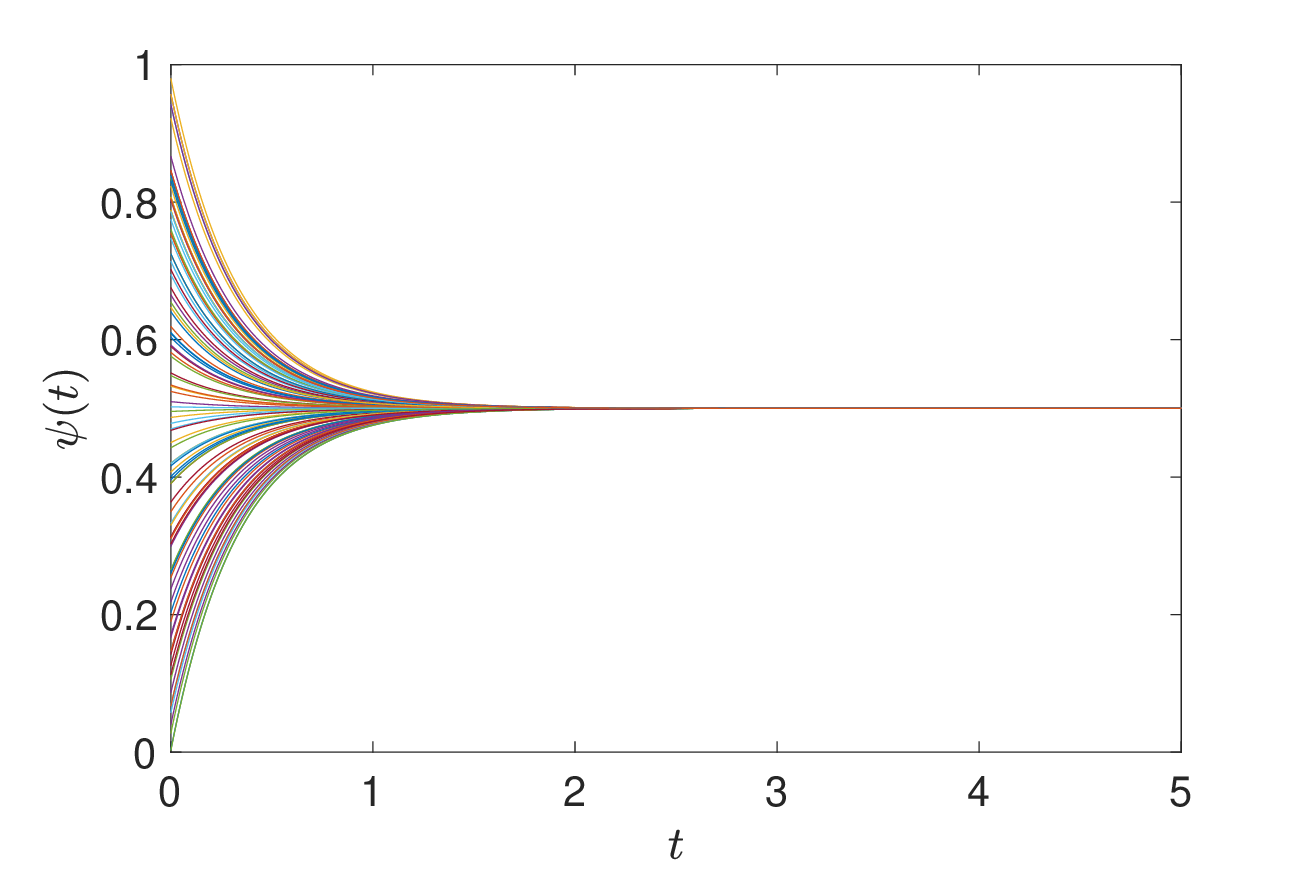}}
\\
{\includegraphics[width=0.24\textwidth]{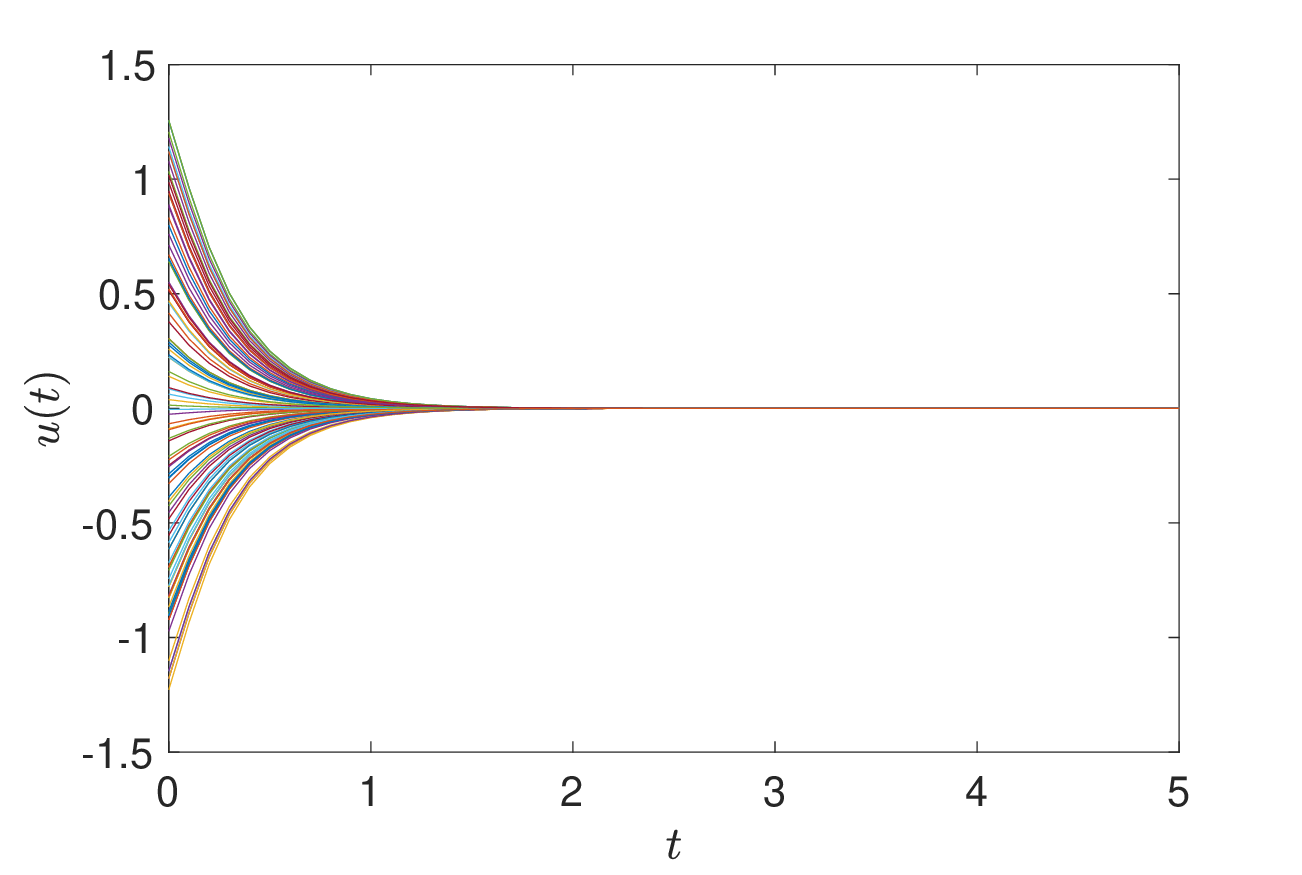}}
{\includegraphics[width=0.24\textwidth]{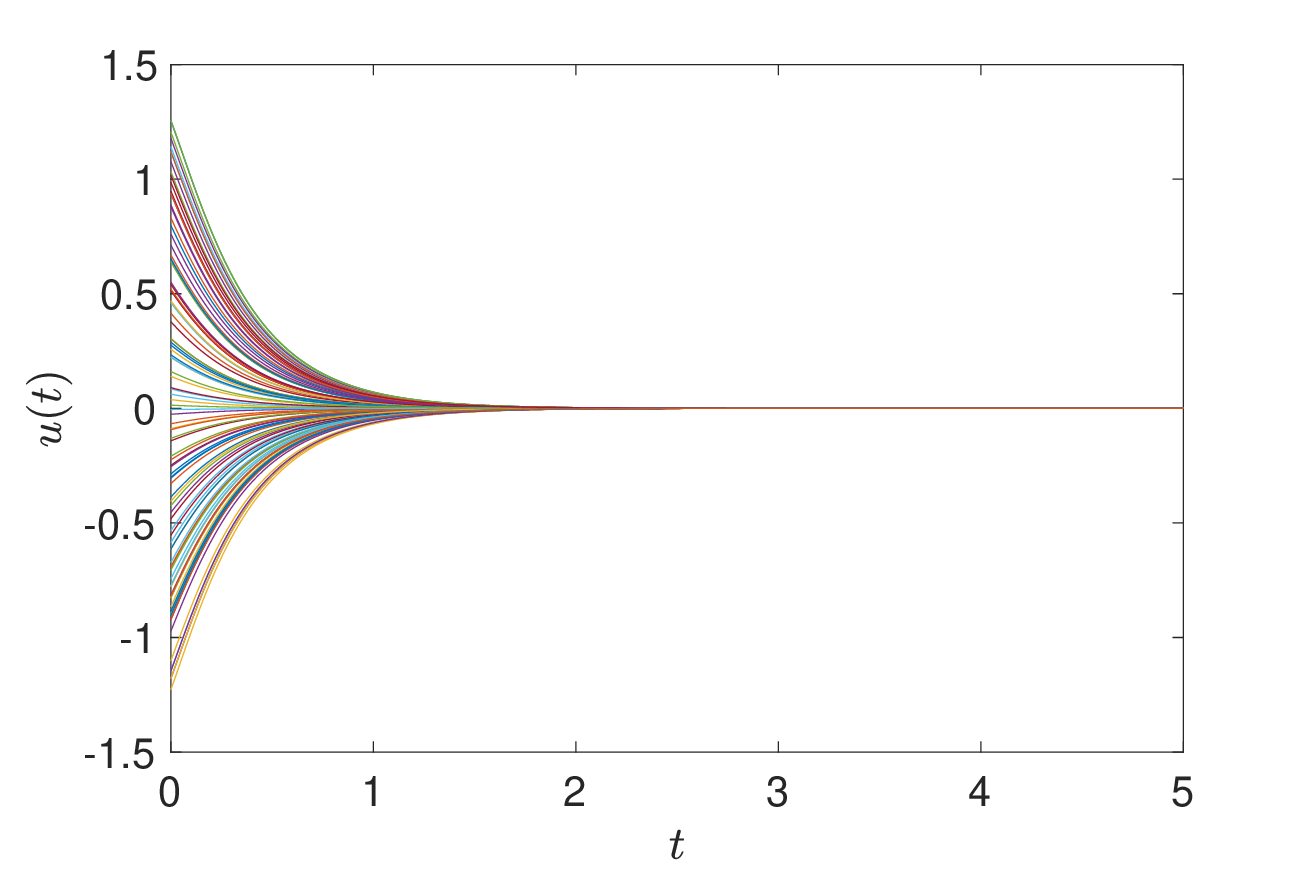}}
{\includegraphics[width=0.24\textwidth]{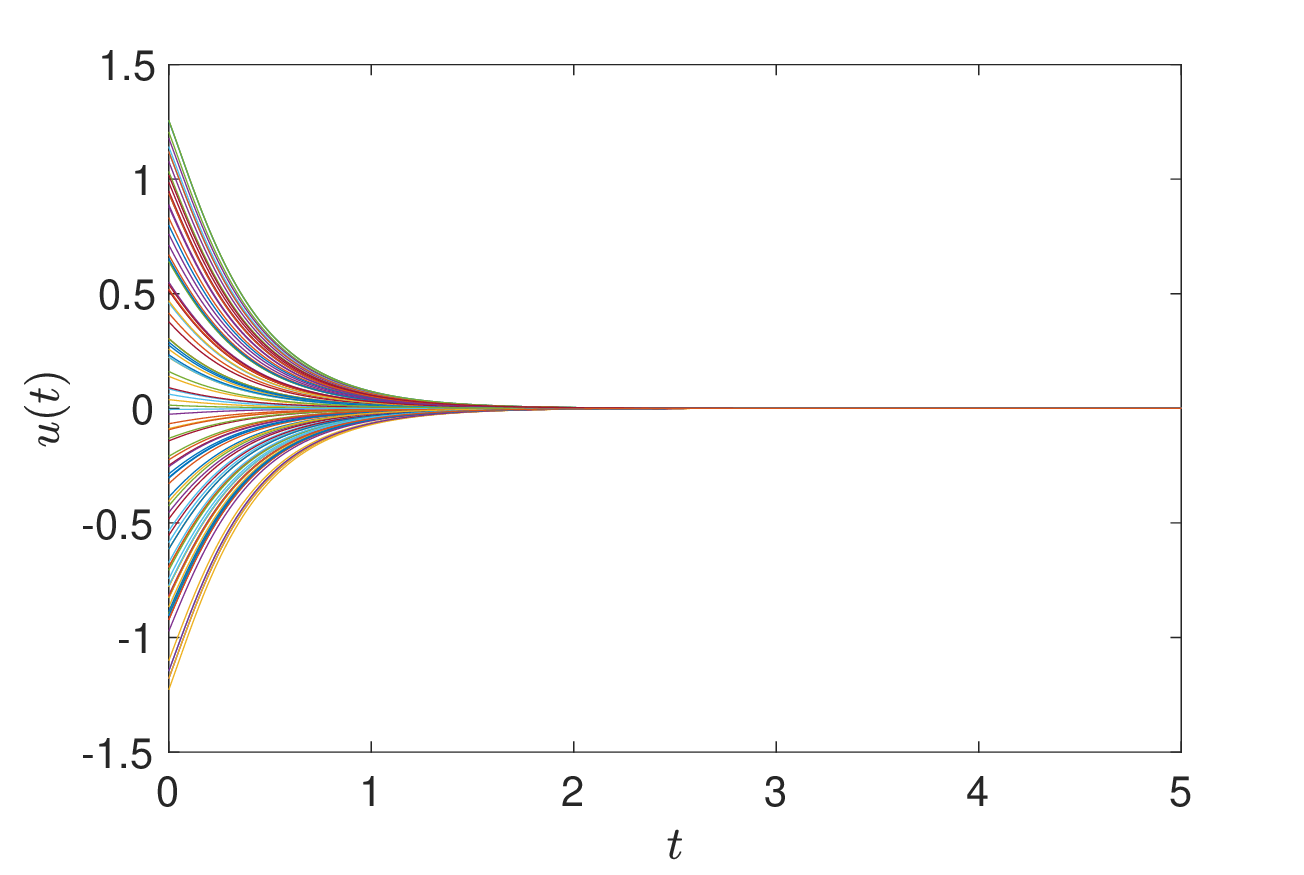}}
\caption{ {\textbf{Test 2.} From left to right, trajectories of state (top) and control (bottom) for discretization levels $h=0.1, 0.01, 0.001$.}}
\label{fig:T2_a}
\end{center}
\end{figure}
\begin{figure}[t]
\begin{center}
{\includegraphics[width=0.4\textwidth]{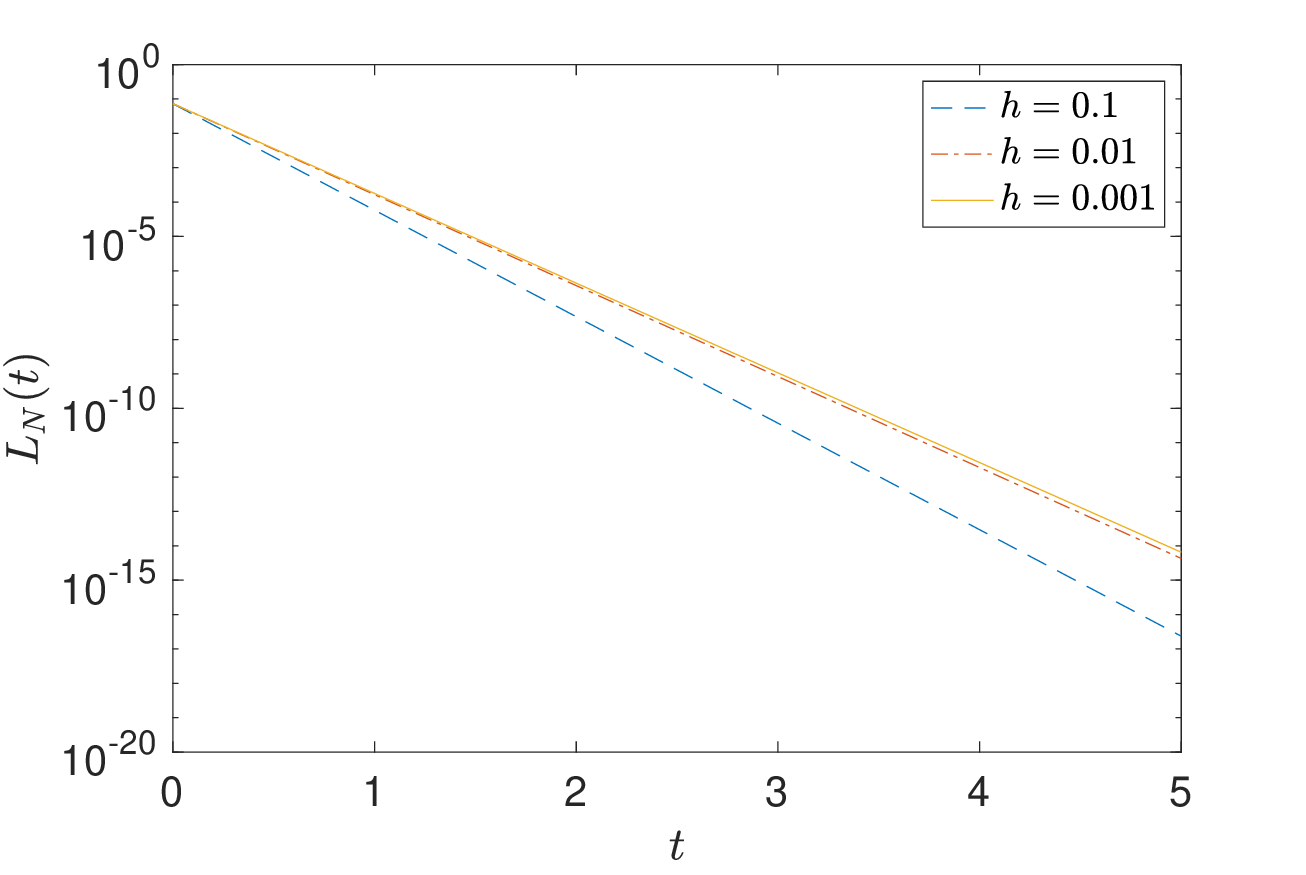}}
\end{center}
\caption{{\textbf{Test 2.} The exponential decay of $L_N$ for different discretization levels $h=0.1, 0.01, 0.001$.}}
\label{fig:T2_b}
\end{figure}

\paragraph{Test 3: Comparison for increasing numbers of agents.}

From the results 
in \cite{main}, uniformity with respect to the parameter $N$ is expected. Also in the case of a time discrete system, we have the emergence of a turnpike phenomenon with interior decay, which remains unaffected by variations in the number of agents $N$. Consequently, our investigation entails the generation of plots for $N$ values ranging from $50$ to $500$, showing consistent patterns across the spectrum.

As in Test 2, we implement the cheap control mechanism with a specified parameter value of $\beta=3$, given its close approximation to the optimal solution. Figure \ref{fig:T3_N1} shows the trajectories of both state and control variables, providing their evolution with increasing $N$ values, namely $50, 100$, and $500$.

In Figure \ref{fig:T3_N2}, a logarithmic scale representation illustrates the exponential decay of $L_N$ across various agent counts. This visualization provides again the effectiveness of our implemented control strategy, even in the presence of a considerable number of agents. In particular, we can observe, also numerically, a very similar behaviour for the evolution of $\psi, u$ and the decay of $L_N$ in time, with different number of agents.

\begin{figure}[t]
\begin{center}
{\includegraphics[width=0.24\textwidth]{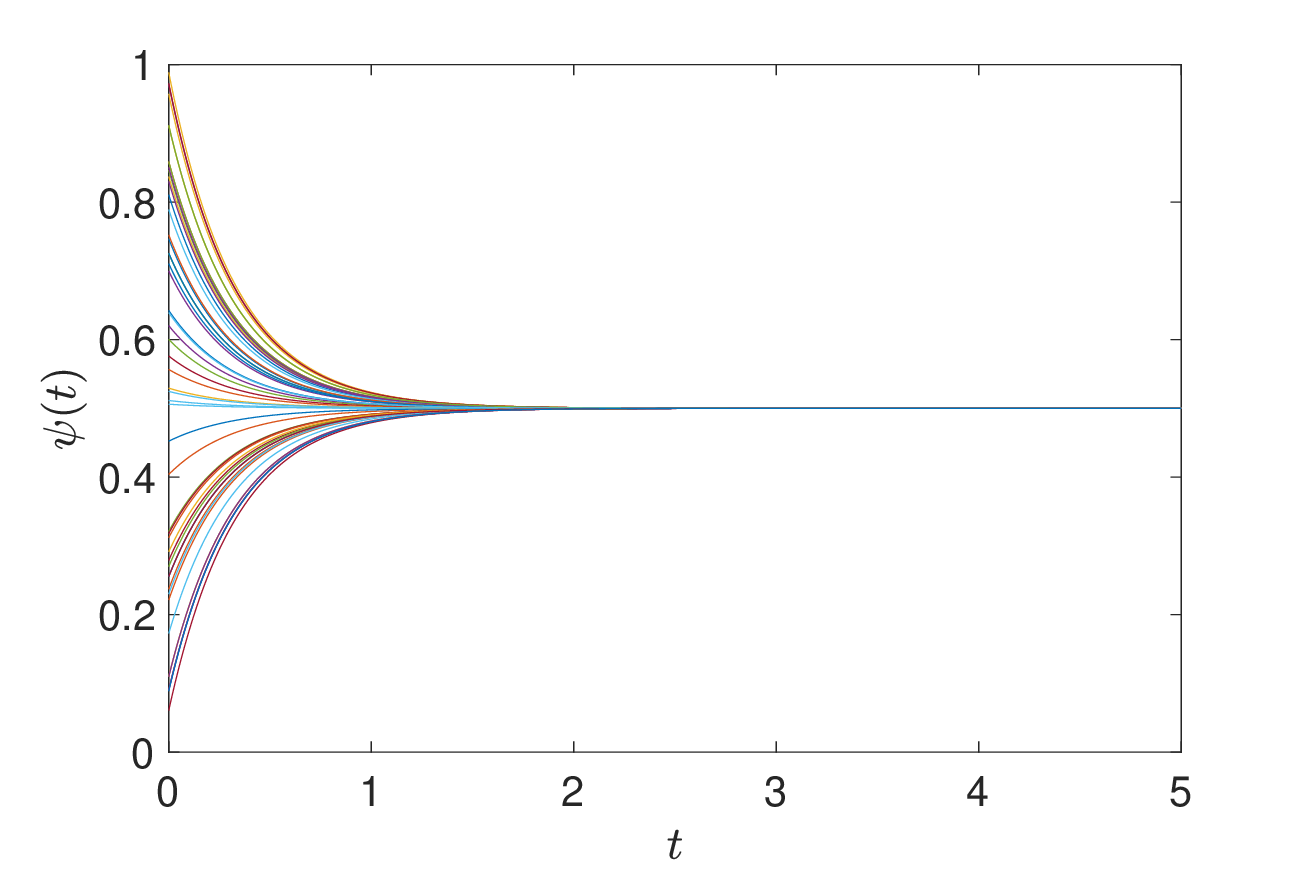}}
{\includegraphics[width=0.24\textwidth]{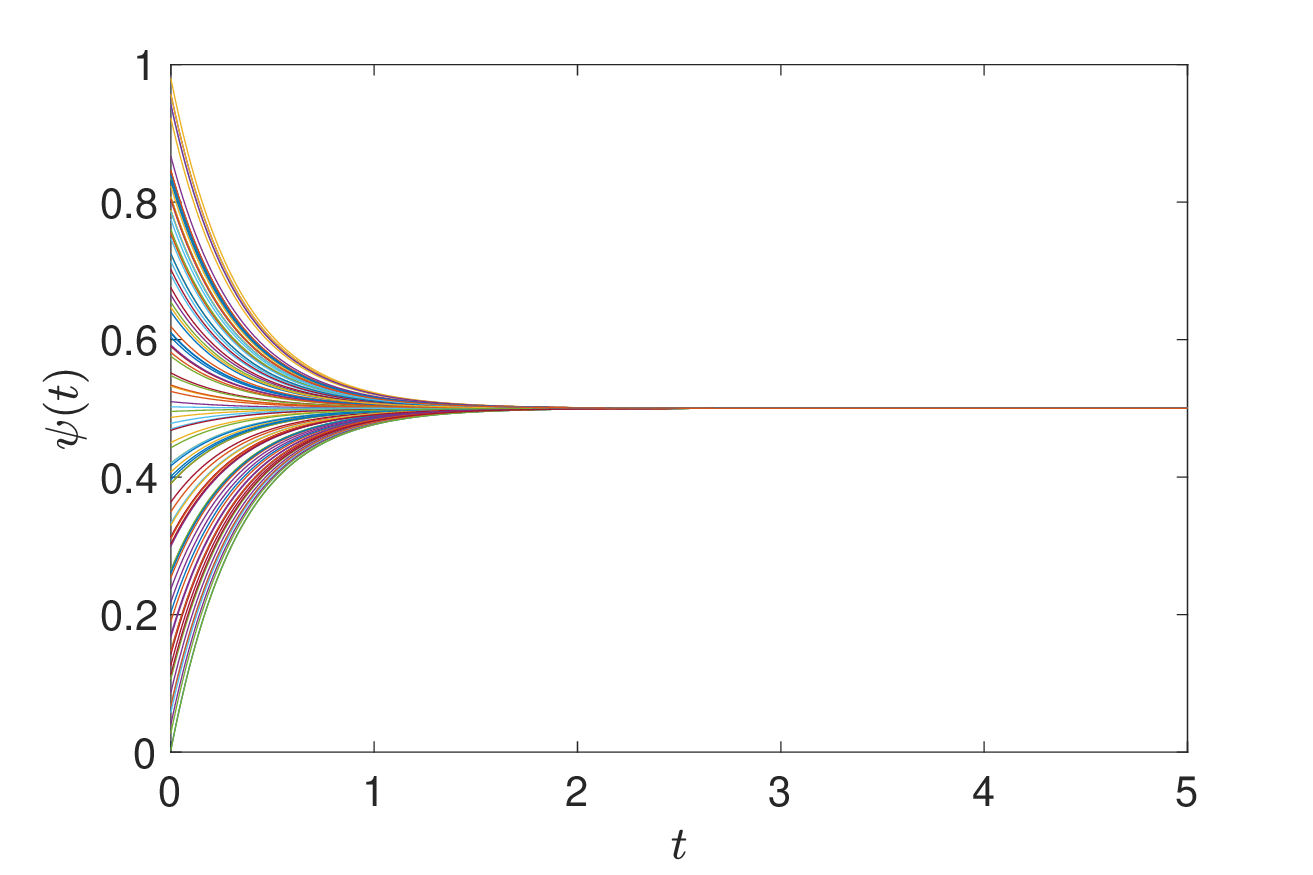}}
{\includegraphics[width=0.24\textwidth]{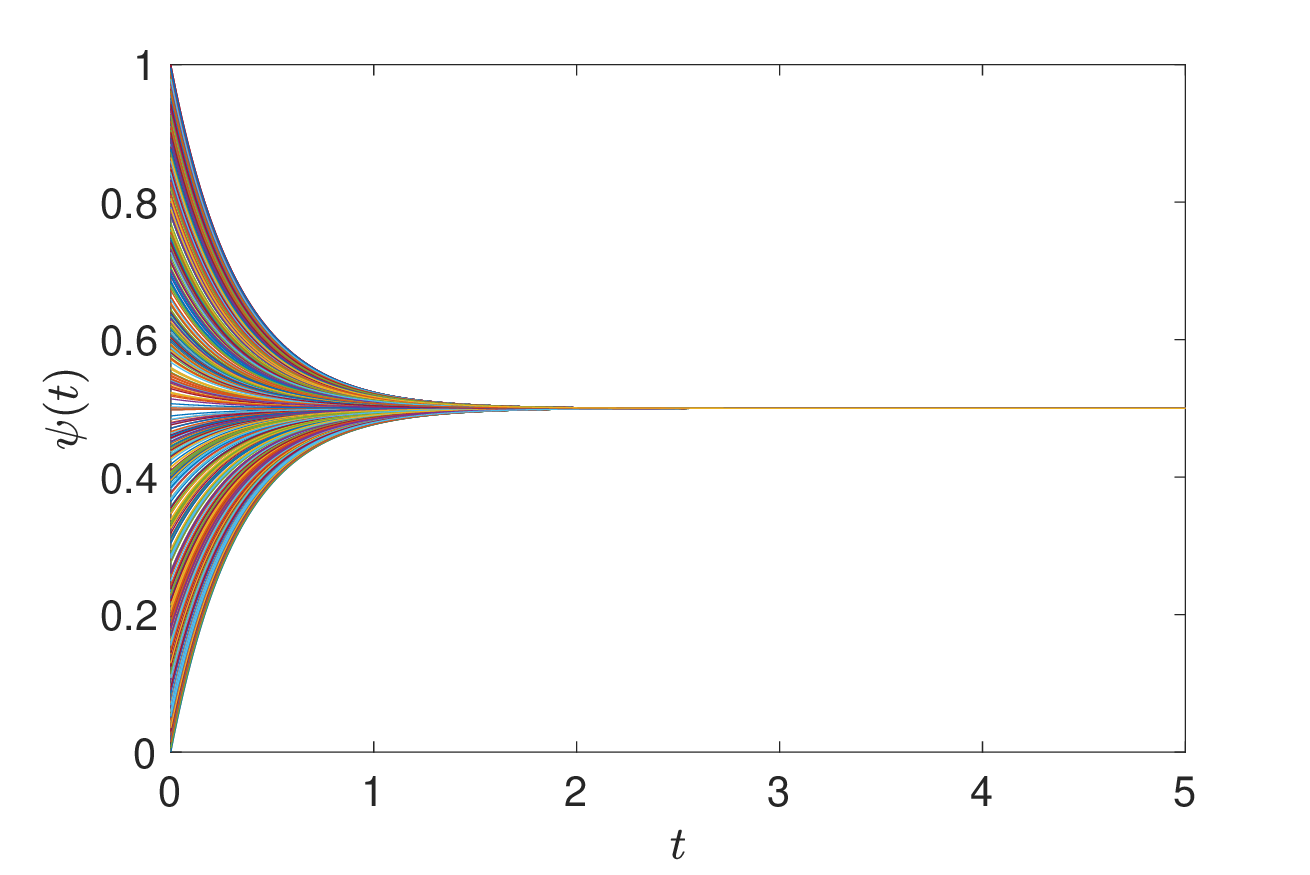}}
\\
{\includegraphics[width=0.24\textwidth]{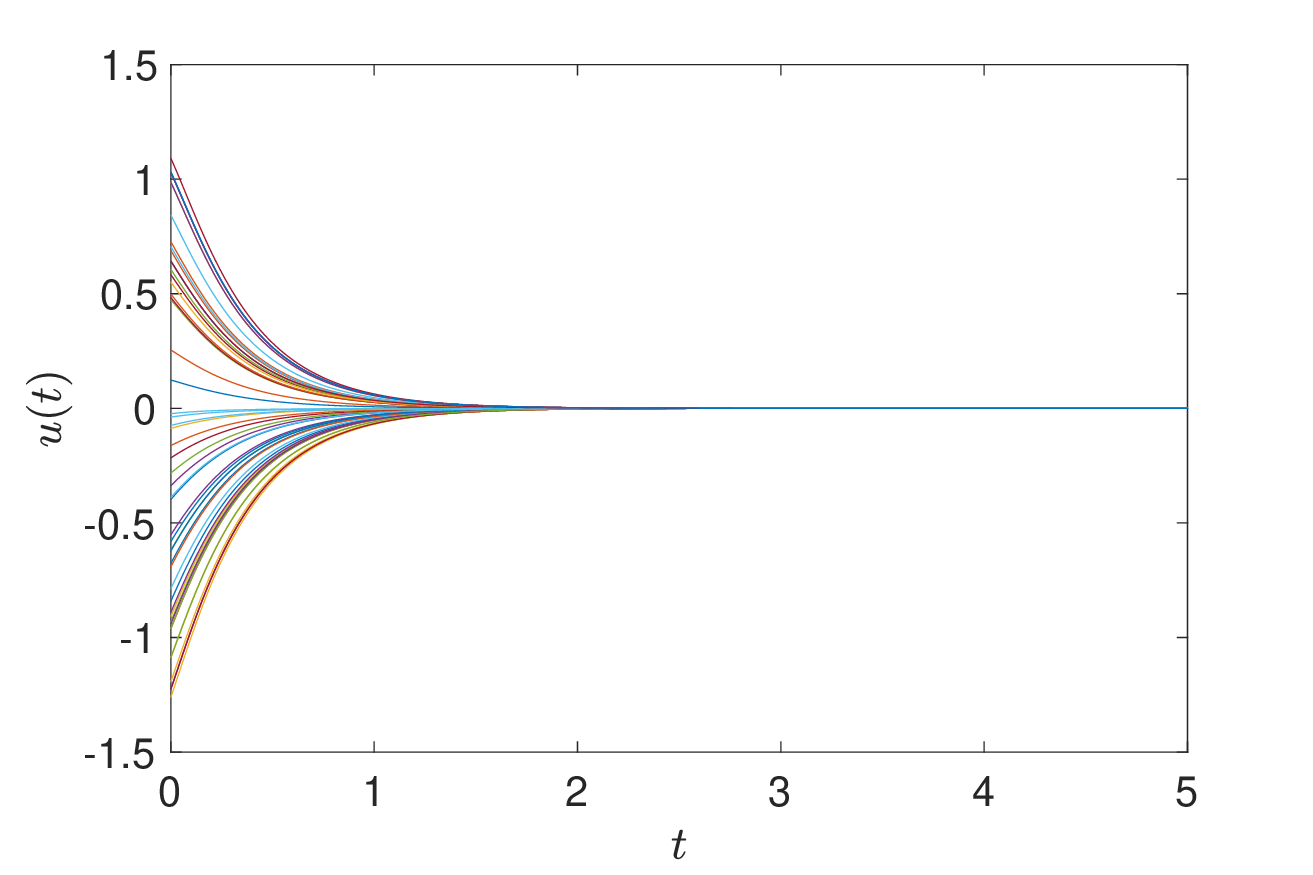}}
{\includegraphics[width=0.24\textwidth]{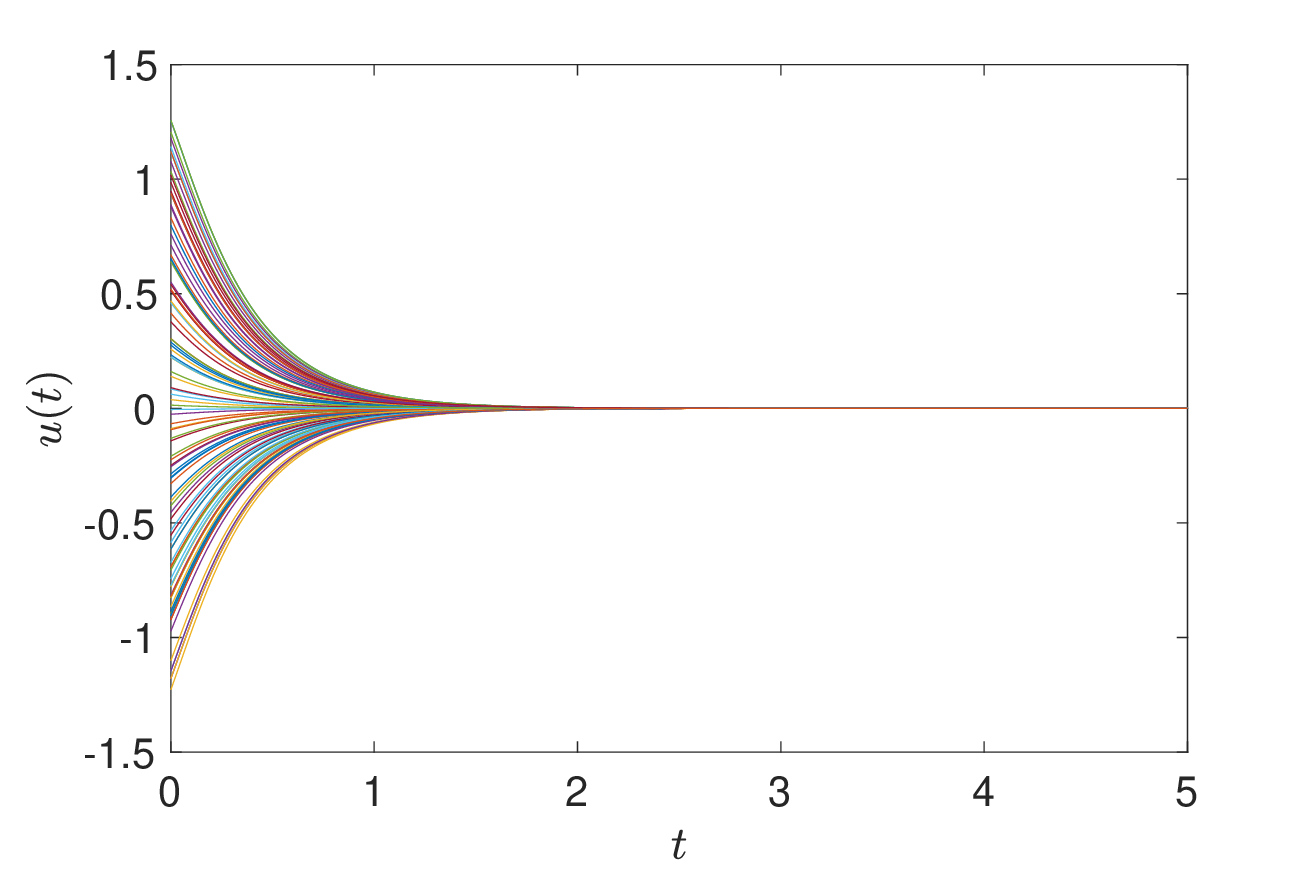}}
{\includegraphics[width=0.24\textwidth]{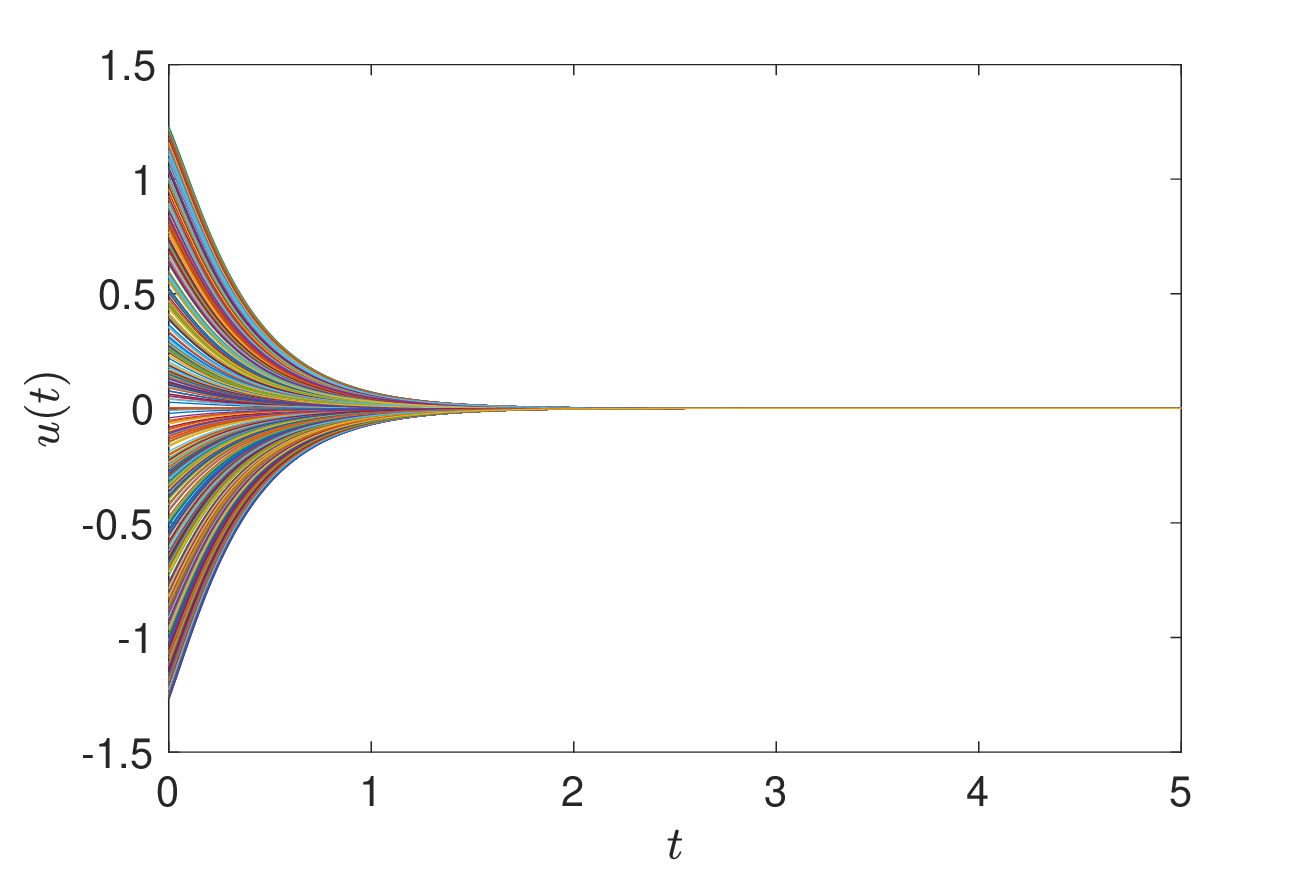}}
\caption{ {\textbf{Test 3.} From left to right, trajectories of state (top) and control (bottom) for an increasing number of agents $N=50, 100, 500$.}}
\label{fig:T3_N1}
\end{center}
\end{figure}
\begin{figure}[t]
\begin{center}
{\includegraphics[width=0.4\textwidth]{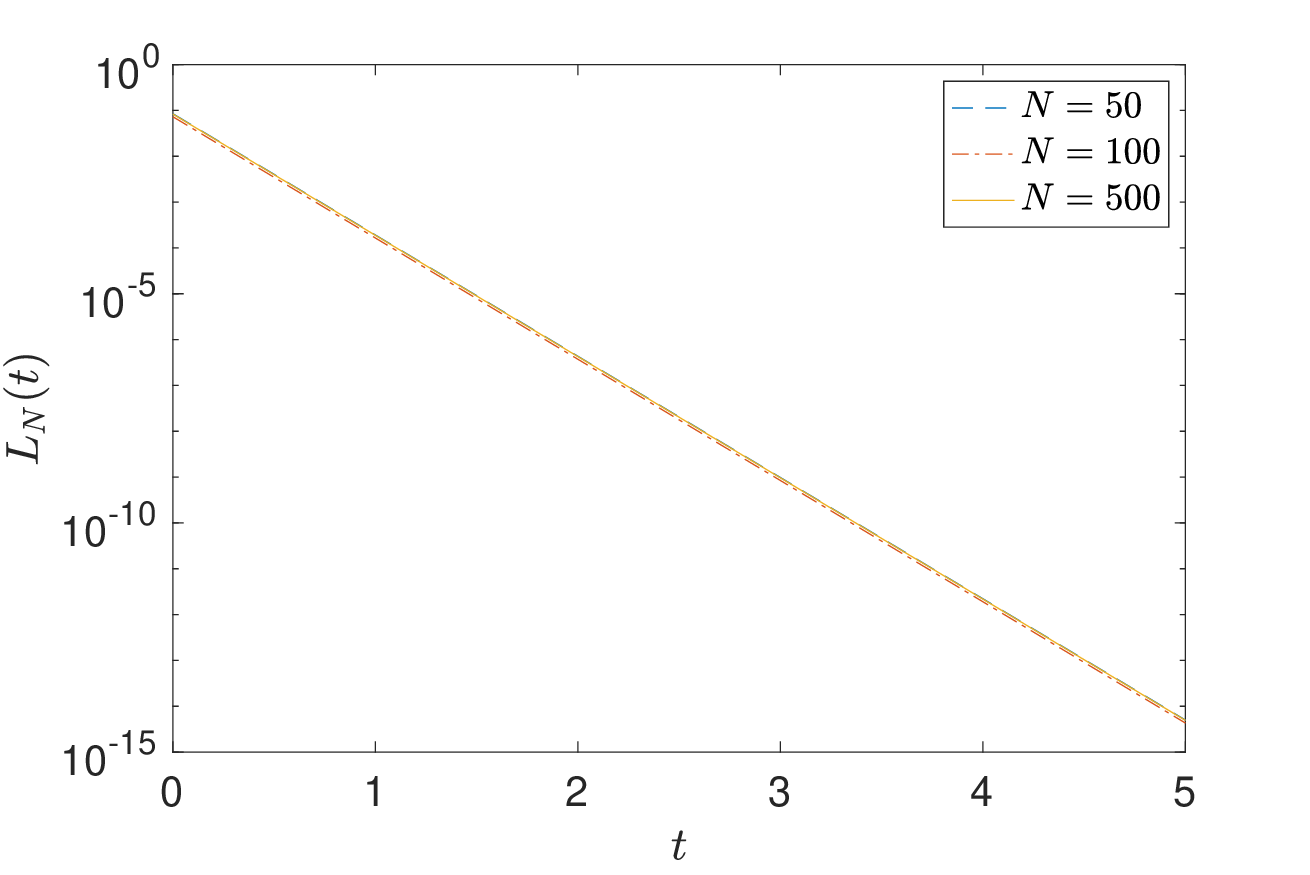}}
\end{center}
\caption{{\textbf{Test 3.} The exponential decay of $L_N$ for increasing number of agents $N=50, 100, 500$.}}
\label{fig:T3_N2}
\end{figure}


\paragraph{Test 4: Non-differentiable interaction kernel.}

Concluding our investigation, we turn our attention to the observation of the turnpike property in systems characterized by a non-differential interaction kernel. Specifically, we consider $P(x,y)=\vert x-y \vert$. Implementing a chep control strategy with $\beta=3$, the trajectories of the controlled system are depicted in Figure \ref{fig:T3_b}. 

Figure \ref{fig:T3_c} provides a complementary examination by showing the exponential decay of the Lyapunov function associated with the uncontrolled and cheap controlled trajectories.

By exploring both trajectory patterns and Lyapunov function decay, we add some insights into the behavior and stability of controlled systems, extending its applicability to scenarios involving non-differential interaction kernels.

\begin{figure}[t]
\begin{center}
	{\includegraphics[width=0.24\textwidth]{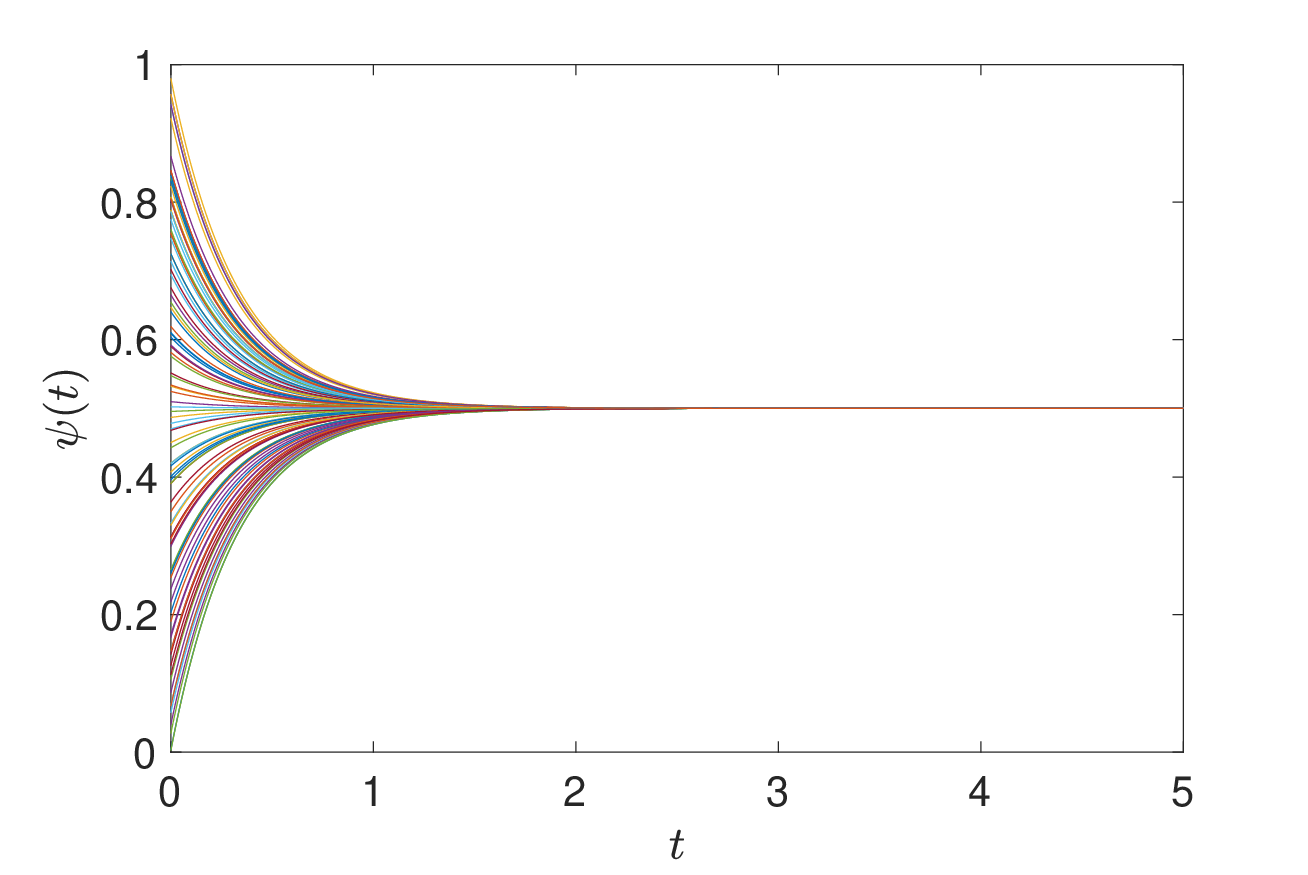}}
 \\
	{\includegraphics[width=0.24\textwidth]{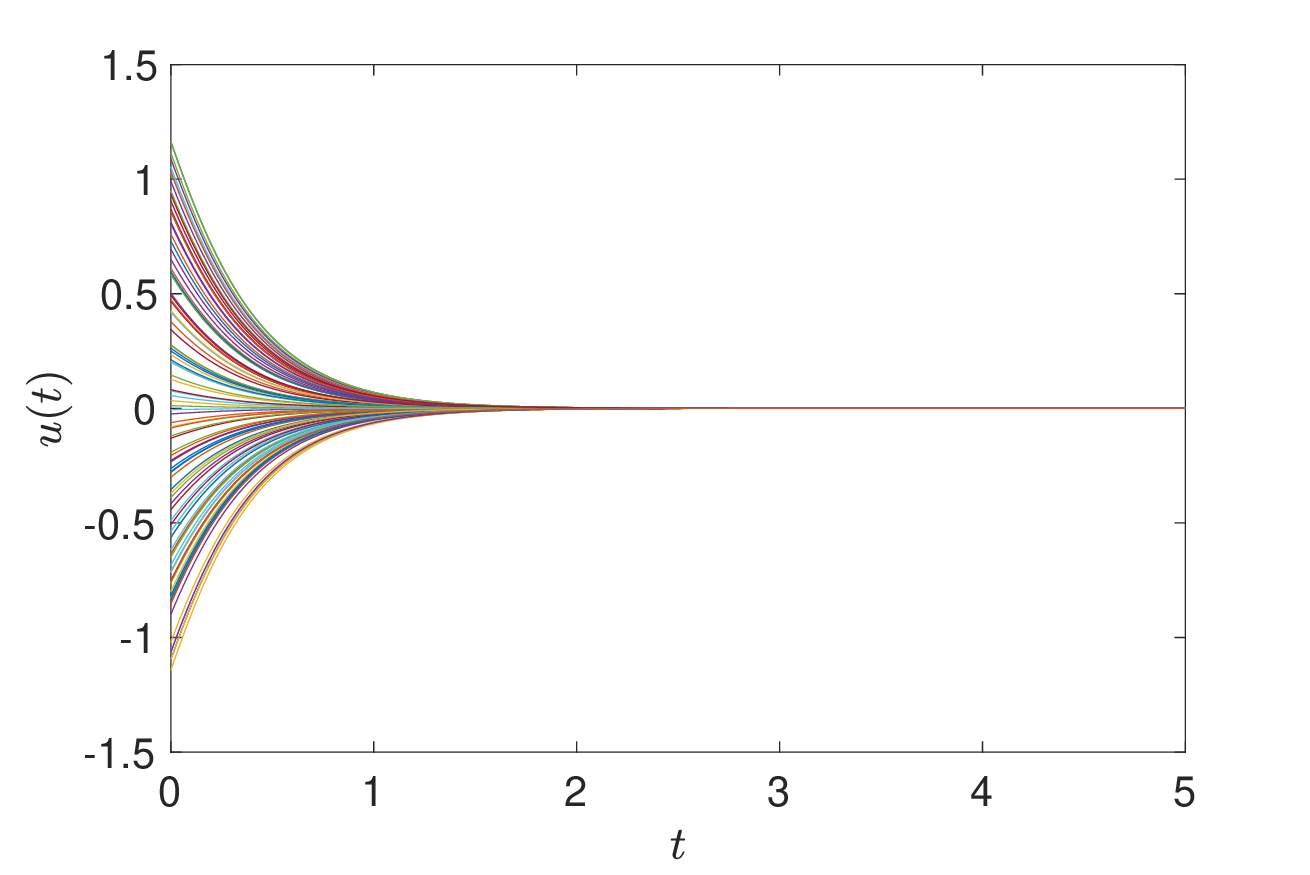}}
	\caption{{\textbf{Test 4.}
			Time evolution of cheap controlled ($\beta=3$) trajectories for a non-differentiable interaction kernel $P(x,y)=\vert x-y \vert$.}}
	\label{fig:T3_b}
\end{center}
\end{figure}
\begin{figure}[t]
\begin{center}
{\includegraphics[width=0.4\textwidth]{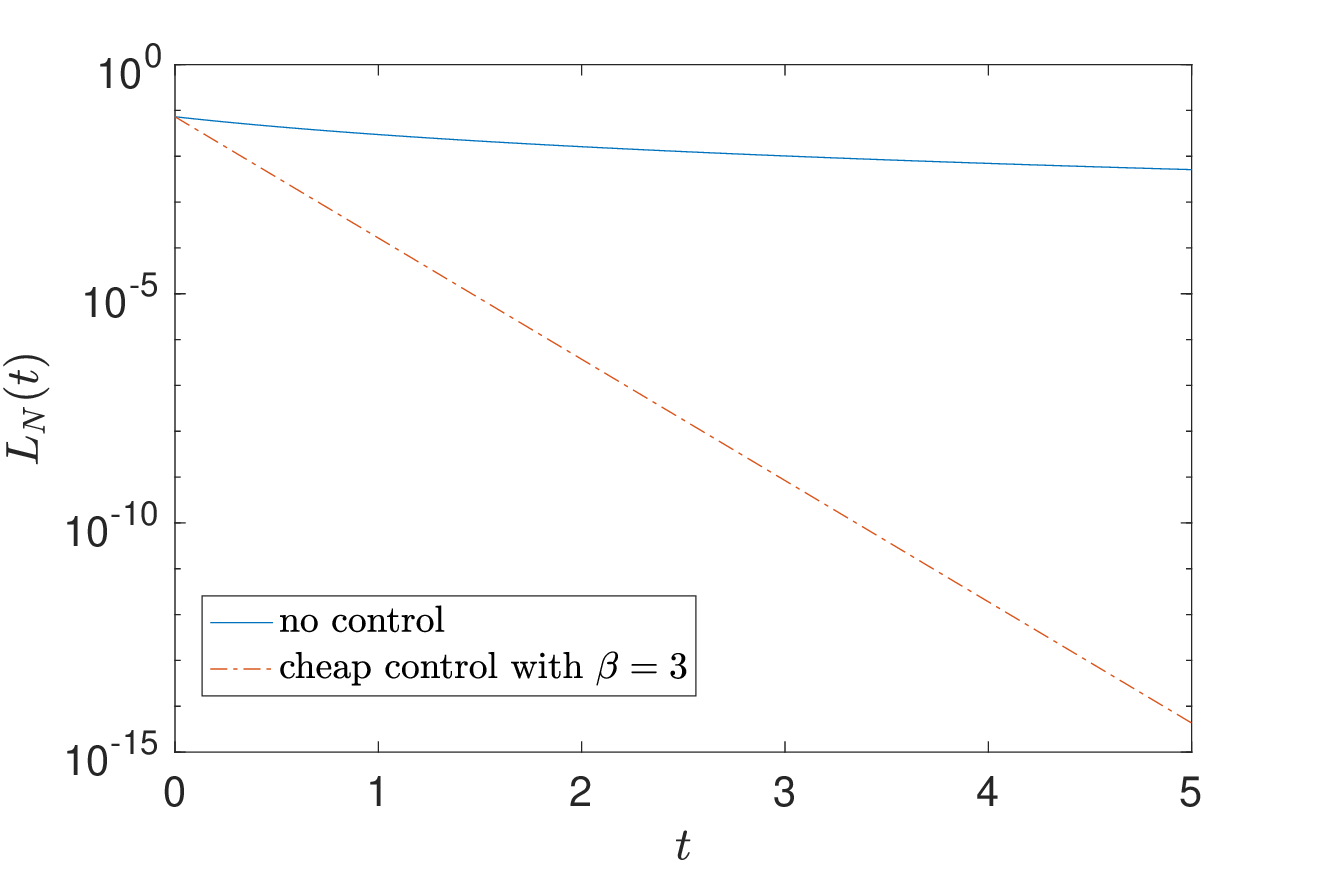}}
\caption{{\textbf{Test 4.} The exponential decay of $L_N$ for uncontrolled and cheap controlled results.}}
\label{fig:T3_c}
\end{center}
\end{figure}

\section{Conclusion}

The continuous-time multi-agent optimal control problem has been discretized and the turnpike property has been established for the discrete-time case. 
Without additional assumptions on the problem, the turnpike property is independent of the step size $h$ of the discretization and therefore holds uniformly in $h.$ This shows that the turnpike property is inherited by suitable numerical discretization schemes.  The proof relies on an extension of the strict dissipativity and the cheap control inequality for the discrete system. Numerical tests  illustrate the turnpike phenomenon and support the theoretical findings. Future research includes the treatment of other integrators in time, in particular, symplectic integrators that are suitable for long time horizons. Another interesting aspect is the time-discrete case for the large agent limit, since we expect that the turnpike property still holds in the limit of infinitely many agents as in \cite{main}.
Also the study of second order models is of interest, see \cite{piccoli2016sparse}.


\subsubsection*{Acknowledgments}
{\small The authors thank the Deutsche Forschungsgemeinschaft (DFG, German Research Foundation) for the financial support through 442047500/SFB1481 within the projects B04 (Sparsity fördernde Muster in kinetischen Hierarchien), B05 (Sparsifizierung zeitabhängiger Netzwerkflußprobleme mittels diskreter Optimierung) and B06 (Kinetische Theorie trifft algebraische Systemtheorie) and through SPP 2298 Theoretical Foundations of Deep Learning  within the Project(s) HE5386/23-1, Meanfield Theorie zur Analysis von Deep Learning Methoden (462234017). C. Segala ia a member of the Italian National Group of Scientific Calculus (Indam GNCS).
This work was also funded by the DFG, TRR 154, \emph{Mathematical Modelling, Simulation and Optimization Using the Example of Gas Networks}, project C03 and C05, Projektnr. 239904186.
Support funding received from the European Union's Horizon Europe research and innovation programme under the Marie Sklodowska-Curie Doctoral Network Datahyking (Grant No. 101072546).
}

\bibliographystyle{siam}
\bibliography{referencesnmr.bib}

\begin{thebibliography}{10}

\bibitem{MR4469721}
{\sc G.~Albi, S.~Bicego, and D.~Kalise}, {\em Gradient-augmented supervised
  learning of optimal feedback laws using state-dependent {R}iccati equations},
  IEEE Control Syst. Lett., 6 (2022), pp.~836--841.

\bibitem{MR4399019}
{\sc G.~Albi, M.~Herty, D.~Kalise, and C.~Segala}, {\em Moment-driven
  predictive control of mean-field collective dynamics}, SIAM J. Control
  Optim., 60 (2022), pp.~814--841.

\bibitem{kinetic}
{\sc G.~Albi, M.~Herty, and L.~Pareschi}, {\em Kinetic description of optimal
  control problems and applications to opinion consensus}, Commun. Math. Sci.,
  13 (2015), pp.~1407--1429.

\bibitem{MR3351435}
\leavevmode\vrule height 2pt depth -1.6pt width 23pt, {\em Kinetic description
  of optimal control problems and applications to opinion consensus}, Commun.
  Math. Sci., 13 (2015), pp.~1407--1429.

\bibitem{MR3894072}
{\sc G.~Albi and L.~Pareschi}, {\em Selective model-predictive control for
  flocking systems}, Commun. Appl. Ind. Math., 9 (2018), pp.~4--21.

\bibitem{use-1}
{\sc B.~D.~O. Anderson and P.~V. Kokotovic}, {\em Optimal control problems over
  large time intervals}, Autom., 23 (1987), pp.~355--363.

\bibitem{application-[4]}
{\sc D.~Armbruster and C.~Ringhofer}, {\em Thermalized kinetic and fluid models
  for reentrant supply chains}, Multiscale Modeling \& Simulation, 3 (2005),
  pp.~782--800.

\bibitem{attraction}
{\sc D.~Balagu{\'e}, J.~Carrillo, T.~Laurent, and G.~Raoul}, {\em Nonlocal
  interactions by repulsive--attractive potentials: Radial ins/stability},
  Physica D: Nonlinear Phenomena, 260 (2013), pp.~5--25.
\newblock Emergent Behaviour in Multi-particle Systems with Non-local
  Interactions.

\bibitem{MR3642940}
{\sc N.~Bellomo, P.~Degond, and E.~Tadmor}, eds., {\em Active particles. {V}ol.
  1. {A}dvances in theory, models, and applications}, Modeling and Simulation
  in Science, Engineering and Technology, Birkh\"{a}user/Springer, Cham, 2017.

\bibitem{MR3969953}
\leavevmode\vrule height 2pt depth -1.6pt width 23pt, eds., {\em Active
  particles. {V}ol. 2. {A}dvances in theory, models, and applications},
  Modeling and Simulation in Science, Engineering and Technology,
  Birkh\"{a}user/Springer, Cham, 2019.

\bibitem{application-[6]}
{\sc N.~BELLOMO and J.~SOLER}, {\em On the mathematical theory of the dynamics
  of swarms viewed as complex systems}, Mathematical Models and Methods in
  Applied Sciences, 22 (2012), p.~1140006.

\bibitem{sparse}
{\sc M.~Bongini, M.~Fornasier, O.~Junge, and B.~Scharf}, {\em Sparse control of
  alignment models in high dimension}, Networks and Heterogeneous Media, 10
  (2015), pp.~647--697.

\bibitem{sparse-Cucker}
{\sc M.~Caponigro, M.~Fornasier, B.~Piccoli, and E.~Tr{\'e}lat}, {\em Sparse
  stabilization and optimal control of the cucker-smale model}, Mathematical
  Control and Related Fields, 3 (2013), pp.~447--466.

\bibitem{turnpike-3}
{\sc D.~A. Carlson, A.~B. Haurie, and A.~Leizarowitz}, {\em Infinite horizon
  optimal control: deterministic and stochastic systems}, Springer Science \&
  Business Media,  (2012).

\bibitem{pureattraction}
{\sc J.~A. Carrillo, M.~DiFrancesco, A.~Figalli, T.~Laurent, and D.~Slep{\v
  c}ev}, {\em {Global-in-time weak measure solutions and finite-time
  aggregation for nonlocal interaction equations}}, Duke Mathematical Journal,
  156 (2011), pp.~229 -- 271.

\bibitem{application-[14]}
{\sc S.~Cordier, L.~Pareschi, and G.~Toscani}, {\em On a kinetic model for a
  simple market economy}, Journal of Statistical Physics, 120 (2005),
  pp.~253--277.

\bibitem{application-[15]}
{\sc I.~Couzin, J.~Krause, N.~Franks, and S.~Levin}, {\em Effective leadership
  and decision-making in animal groups on the move}, Nature, 433 (2005),
  pp.~513--6.

\bibitem{MR4599948}
{\sc T.~Faulwasser and L.~Gr\"{u}ne}, {\em Turnpike properties in optimal
  control. {A}n overview of discrete-time and continuous-time results}, in
  Numerical control. {P}art {A}, vol.~23 of Handb. Numer. Anal., North-Holland,
  Amsterdam, [2022] \copyright 2022, pp.~367--400.

\bibitem{MR4493557}
{\sc T.~Faulwasser, L.~Gr\"{u}ne, J.-P. Humaloja, and M.~Schaller}, {\em The
  interval turnpike property for adjoints}, Pure Appl. Funct. Anal., 7 (2022),
  pp.~1187--1207.

\bibitem{receding}
{\sc L.~Gr{\"u}ne}, {\em Economic receding horizon control without terminal
  constraints}, Autom., 49 (2013), pp.~725--734.

\bibitem{MR4402854}
{\sc L.~Gr\"{u}ne}, {\em Dissipativity and optimal control: examining the
  turnpike phenomenon}, IEEE Control Syst., 42 (2022), pp.~74--87.

\bibitem{use-2}
{\sc L.~Gr{\"u}ne and M.~Stieler}, {\em Asymptotic stability and transient
  optimality of economic mpc without terminal conditions}, Control. Bd., 24
  (2014), pp.~1187--1196.

\bibitem{gruene2018turnpike}
{\sc L.~Gru\"une and R.~Guglielmi}, {\em Turnpike properties and strict
  dissipativity for discrete time linear quadratic optimal control problems},
  SIAM Journal on Control and Optimization, 56 (2018), pp.~1282--1302.

\bibitem{interior-original}
{\sc M.~Gugat}, {\em On the turnpike property with interior decay for optimal
  control problems}, Mathematics of Control, Signals, and Systems, 33 (2021),
  pp.~1--22.

\bibitem{main}
{\sc M.~Gugat, M.~Herty, and C.~Segala}, {\em The turnpike property for
  mean-field optimal control problems}, European Journal of Applied
  Mathematics,  (2023), pp.~1--15.

\bibitem{gugat2023optimal}
{\sc M.~Gugat and M.~Lazar}, {\em Optimal control problems without terminal
  constraints: The turnpike property with interior decay}, International
  Journal of Applied Mathematics and Computer Science, 33 (2023), pp.~429--438.

\bibitem{MR3431287}
{\sc M.~Herty, L.~Pareschi, and S.~Steffensen}, {\em Mean-field control and
  {R}iccati equations}, Netw. Heterog. Media, 10 (2015), pp.~699--715.

\bibitem{application-[29]}
{\sc M.~Herty and C.~Ringhofer}, {\em Averaged kinetic models for flows on
  unstructured networks}, Kinetic and Related Models, 4 (2011).

\bibitem{piccoli2016sparse}
{\sc B.~Piccoli, F.~Rossi, and E.~Tr{\'e}lat}, {\em Sparse control of
  second-order cooperative systems and partial differential equations to
  approximate alignment}, in 22nd International Symposium on Mathematical
  Theory of Networks and Systems, vol.~2016, 2016.

\bibitem{use-1.1}
{\sc A.~M. Sahlodin and P.~I. Barton}, {\em Optimal campaign continuous
  manufacturing}, Industrial \& Engineering Chemistry Research, 54 (2015),
  pp.~11344--11359.

\bibitem{MR4288153}
{\sc A.~Tosin and M.~Zanella}, {\em Uncertainty damping in kinetic traffic
  models by driver-assist controls}, Math. Control Relat. Fields, 11 (2021),
  pp.~681--713.

\bibitem{measure}
{\sc E.~Trelat and C.~Zhang}, {\em Integral and measure-turnpike properties for
  infinite-dimensional optimal control systems}, Mathematics of Control,
  Signals, and Systems, 30 (2018).

\bibitem{IPOPT}
{\sc A.~W{\"a}chter and L.~Biegler}, {\em On the implementation of an
  interior-point filter line-search algorithm for large-scale nonlinear
  programming}, Mathematical Programming, 106 (2006), pp.~25--57.

\bibitem{turnpike-2}
{\sc A.~J. Zaslavski}, {\em Turnpike Properties in the Calculus of Variations
  and Optimal Control}, Nonconvex Optimization and Its Applications, 80,
  Springer US, New York, NY, 1st ed. 2006.~ed., 2006.

\bibitem{MR3973342}
\leavevmode\vrule height 2pt depth -1.6pt width 23pt, {\em Necessary and
  sufficient turnpike conditions}, Pure Appl. Funct. Anal., 4 (2019),
  pp.~463--476.

\bibitem{turnpike-1}
\leavevmode\vrule height 2pt depth -1.6pt width 23pt, {\em Turnpike conditions
  in infinite dimensional optimal control}, vol.~148 of Springer Optimization
  and Its Applications, Springer, Cham, 2019.

\end{thebibliography}
	
\end{document}